\documentclass[12pt,reqno]{amsart}

\title[The Schr\"odinger problem]{A survey of the Schr\"odinger problem and some of its  connections with optimal transport}
\author{Christian L\'eonard}
\date{Revised version. March 7, 2013}
\thanks{Author partially supported by the ANR project GeMeCoD. ANR 2011 BS01 007 01}

\usepackage{amssymb, amsmath, amsfonts, latexsym, enumerate}
\usepackage[usenames,dvipsnames]{pstricks}
\usepackage{epsfig}
\usepackage[german, french, english]{babel}
\usepackage[utf8]{inputenc}

\newtheorem{theorem}{Theorem}
\newtheorem{lemma}[theorem]{Lemma}
\newtheorem{proposition}[theorem]{Proposition}

\newtheorem{claim}[theorem]{Claim}
\newtheorem{statement}[theorem]{Informal statement}
\newtheorem{problem}[theorem]{Problem}

\newtheorem{definition}[theorem]{Definition}

\newtheorem{assumption}[theorem]{Assumption}

\theoremstyle{remark}
\newtheorem{remark}[theorem]{Remark}

\numberwithin{theorem}{section}
\numberwithin{equation}{section}

\newcommand{\RR}{\mathbb{R}}
\newcommand{\Rn}{\mathbb{R}^n}

\newcommand{\1}{\mathbf{1}}

\newcommand{\ttimes}{\!\times\!}

\newcommand\pf{_{\#}}
\newcommand{\vol}{\mathrm{vol}}

\newcommand{\as}{\textrm{-a.s.}}
\renewcommand{\ae}{\textrm{-a.e.}}

\newcommand{\scal}{\!\cdot\!}

    \DeclareMathOperator{\dom}{dom}
	\DeclareMathOperator{\icor}{icor}

\newcommand{\boulette}[1]{$\bullet$\ Proof of #1.}
\newcommand{\Boulette}[1]{\par\medskip\noindent $\bullet$\ Proof of #1.}

\newcommand\Lim[1]{\lim_{#1\rightarrow\infty}}
\newcommand\Liminf[1]{\liminf_{#1\rightarrow\infty}}

\newcommand\Glim[1]{\Gamma\textrm{-}\lim_{#1\rightarrow\infty}}

\newcommand{\cadlag}{càdlàg}
\newcommand{\ud}{\frac{1}{2}}

\newcommand\XX{\mathcal{X}}
\newcommand\XXX{\XX^2}
\newcommand\PX{\mathrm{P}(\XX)}
\newcommand\PdX{\mathrm{P}_2(\XX)}
\newcommand\PXX{\mathrm{P}(\XXX)}
\newcommand\MX{\mathrm{M}_+(\XX)}
\newcommand\MXX{\mathrm{M}_+(\XXX)}
\newcommand\PY{\mathrm{P}(Y)}
\newcommand\PO{\mathrm{P}(\Omega)}
\newcommand\MO{\mathrm{M}_+(\Omega)}

\newcommand\OO{\Omega}

\newcommand\ii{{[0,1]}}
\newcommand\iX{{[0,1]\times\XX}}
\newcommand\IX{\int_{\XX}}
\newcommand\IXX{\int_{\XXX}}
\newcommand\IO{\int_\Omega}
\newcommand\Iii{\int_\ii}
\newcommand\IiX{\int_\iX}
\newcommand\IiXX{\int_{\ii\times\XXX}}

\newcommand\Ph{\widehat{P}}
\newcommand\ph{\widehat{\pi}}
\newcommand{\Rt}{\widetilde{R}}

\newcommand{\Lf}{\overrightarrow{L}}
\newcommand{\Lb}{\overleftarrow{L}}
\newcommand{\Af}{\overrightarrow{A}}
\newcommand{\Ab}{\overleftarrow{A}}

\begin{document}

% ************** page de garde ******************************

 \address{Modal-X. Université Paris Ouest. B\^at.\! G, 200 av. de la République. 92001 Nanterre, France}
 \email{christian.leonard@u-paris10.fr}

\begin{abstract} 
This article is aimed at presenting the Schr\"odinger problem and some of its  connections with optimal transport. We hope that it can be used as a basic user's guide to Schr\"odinger problem. We also give a survey of the related literature. In addition, some new results are proved.
\end{abstract}

\maketitle 
\tableofcontents

% ************* corps du texte ****************************

\section{Introduction}

This article is aimed at presenting the Schr\"odinger problem and some of its  connections with optimal transport. We hope that it can be used as a basic user's guide to Schr\"odinger problem. We also give a survey of the related literature. In addition, some new results are proved.

We denote by $\PY$ and $\mathrm{M}_+(Y)$ the sets of all probability and positive measures on a space $Y.$

In 1931, Schr\"odinger \cite{Sch31,Sch32} addressed a problem which is translated in modern terms\footnote{Recall that in the early 30's the axioms of probability theory were not  settled and that the notion of random process still relied entirely on physical intuition.} as follows.
 Let $\XX=\Rn$ or more generally  a  complete connected Riemannian manifold without boundary, $\OO=C(\ii,\XX)$ be the space of all continuous $\XX$-valued paths    on the unit time interval $\ii$ and  denote $R\in\MO$ the law of the reversible Brownian motion on $\XX,$ i.e.\ the Brownian motion with the volume measure as its initial distribution. Remark that $R$ is an unbounded measure on $\OO$ whenever the manifold $\XX$ is not compact. Define the relative entropy of any probability measure $P$ with respect to $R$ by
\begin{equation*}
H(P|R)=\IO \log \left(\frac{dP}{dR}\right) \,dP\in(-\infty,\infty],\qquad P\in\PO
\end{equation*}
if $P$ is absolutely continuous with respect to $R$ and the above integral is meaningful, and $H(P|R)=\infty$ otherwise. A precise definition of the relative entropy with respect to an unbounded measure $R$ is presented at the Appendix.  The \emph{dynamic Schr\"odinger  problem} is
\begin{equation}\label{sdyn}
H(P|R)\to \textrm{min};\qquad P\in\PO: P_0=\mu_0, P_1=\mu_1,
\tag{S$_{\textrm{dyn}}$}
%\tag{\textsf{s$_{\textrm{dyn}}$}}
\end{equation}
where $\mu_0, \mu_1\in\PX$ are  prescribed values of the initial and final time marginals $P_0:=P(X_0\in \cdot)$ and $P_1:=P(X_1\in\cdot)$ of $P.$ Here $(X_t)_{0\le t\le1}$ is the canonical process on $\OO.$ 

This is a convex minimization problem since $H(\cdot|R)$ is a convex function and the constraint set $\left\{P\in \PO:P_0=\mu_0,P_1=\mu_1\right\} =\left\{P\in \mathrm{M}(\OO): P\ge0,P_0=\mu_0,P_1=\mu_1\right\} $ is a convex subset of the vector space $\mathrm{M}(\OO)$ of all bounded measures on $\OO.$ Furthermore, as $H(\cdot|R)$ is \emph{strictly} convex, if \eqref{sdyn} admits a solution, it must be unique. Let $\Ph\in\PO$ be this solution (if it exists). We shall see at Proposition \ref{res-03} that it disintegrates as 
\begin{equation}\label{eq-37}
\Ph(\cdot)=\IXX R ^{xy}(\cdot)\,\ph(dxdy)
\end{equation}
where for all $x,y\in\XX,$ $R ^{xy}:=R(\cdot\mid X_0=x,X_1=y)$ is the Brownian bridge from $x$ to $y$ and  $\ph\in \mathrm{P}(\XX\times\XX)$
 is the unique solution to the following \emph{static Schr\"odinger problem}
\begin{equation}\label{s}
H(\pi|R _{01})\to \textrm{min};\qquad \pi\in \PXX:\pi_0=\mu_0,\pi_1=\mu_1.
\tag{S}
%\tag{\textsf{s}}
\end{equation}
Here,
\begin{equation}\label{eq-40}
R _{01}(dxdy):=R((X_0,X_1)\in dxdy) \propto \exp(-d(x,y)^2/2)\,\textrm{vol} (dx)\textrm{vol} (dy)
\end{equation}
is the joint law of the initial and final positions of the reversible Brownian motion $R$, $d$ is the Riemannian distance  and $\pi_0:=\pi(\cdot\times\XX)$ and $\pi_1:=\pi(\XX\times\cdot)$ are the first and second marginals of $\pi\in\PXX.$

The disintegration formula \eqref{eq-37} means that $\Ph$ shares its bridges with $R,$ that is: $\Ph ^{xy}=R ^{xy}$ for almost all $x,y,$ and that this mixture of bridges is governed by the unique solution $\ph$ to the static Schr\"odinger problem \eqref{s}. It also follows from \eqref{eq-37} that the values of the dynamic and static problems are equal: $\inf \eqref{sdyn}=\inf \eqref{s}.$

The structure of problem \eqref{s} is similar to Monge-Kantorovich problem's one:
\begin{equation}\label{mk}
\IXX c(x,y)\,\pi(dxdy)\to \textrm{min};\qquad \pi\in \PXX:\pi_0=\mu_0,\pi_1=\mu_1
\tag{MK}
%\tag{\textsf{mk}}
\end{equation}
where $c:\XXX\to[0,\infty)$ represents the cost for transporting a unit mass from the initial location $x$ to the final location $y$. Both are convex optimization problems, but unlike \eqref{s}, the  linear program \eqref{mk} might admit infinitely many solutions.  Since \eqref{eq-40} writes as $R _{01}(dxdy)\propto\exp\big(-c(x,y)\big)\,\textrm{vol} (dx)\textrm{vol} (dy)$ with
\begin{equation*}
c(x,y)=d^2(x,y)/2,\quad x,y\in\XX,
\end{equation*}
it happens that the Schr\"odinger problem \eqref{s} is connected to the quadratic Monge-Kantorovich optimal transport problem  \eqref{mk} which is specified by this quadratic cost function. The natural dynamic version of \eqref{mk} is
\begin{equation}\label{mkdyn}
\IO C\,dP\to \textrm{min};\qquad P\in\PO: P_0=\mu_0, P_1=\mu_1
\tag{MK$_{\mathrm{dyn}}$}
%\tag{\textsf{mk}$_{\mathrm{dyn}}$}
\end{equation}
with 
\begin{equation}\label{eq-43}
C(\omega)=\Iii |\dot \omega_t|_{\omega_t}^2/2\,dt\in[0,\infty],\quad \omega\in\OO,
\end{equation}
where we put $C(\omega)=\infty$ when $\omega$ is not absolutely continuous.
\\
Let us comment on the choice of this dynamic version of \eqref{mk}. For all $x,y\in\XX,$ we have 
\begin{equation}\label{eq-46}
c(x,y)=\inf_{}\{C(\omega);\omega\in\OO:\omega_0=x,\omega_1=y\}
\end{equation}
and this infimum is attained at the  constant speed geodesic path $\gamma ^{xy}$ between $x$ and $y,$  which is assumed to be unique for any $(x,y),$ for simplicity. Therefore, the solutions of \eqref{mk} and \eqref{mkdyn} are in one-one correspondence: 
\begin{itemize}
\item
If $\Ph$ solves \eqref{mkdyn}, then $\Ph _{01}:=\Ph((X_0,X_1)\in\cdot)$ solves \eqref{mk};
\item
If $\ph$ solves \eqref{mk}, then the solution of \eqref{mkdyn} is
\begin{equation}\label{eq-42}
\Ph(\cdot)=\IXX \delta _{\gamma ^{xy}}(\cdot)\,\ph(dxdy)
\end{equation}
where $\delta_a$ denotes the Dirac measure at $a$.
\end{itemize}
Furthermore, we have equality of the values of the problems: $\inf\eqref{mk}=\inf\eqref{mkdyn}\in [0,\infty].$

Again, the dynamic Schr\"odinger problem \eqref{sdyn} and the dynamic Monge-Kantorovich problem \eqref{mkdyn} are similar. Comparing their respective solutions \eqref{eq-37} and \eqref{eq-42}, we see that the $\ph$'s solve their respective static problem \eqref{s} and \eqref{mk}, while for each $(x,y)$ the bridge $R ^{xy}\in\PO$ in \eqref{eq-37} plays the role of $\delta _{\gamma ^{xy}}\in\PO$ in \eqref{eq-42}.

All the notions pertaining to the Monge-Kantorovich optimal transport problems \eqref{mk} and \eqref{mkdyn} which are going to be invoked below are discussed in great details in C.~Villani's textbook \cite{Vill09}.

\subsection*{Displacement interpolations in $\PdX$}

Let $\PX$ denote the set of all probability measures on $\XX$ and $\PdX:=\left\{p\in\PX: \IX d^2(x_0,y)\,p(dy)<\infty\right\}.$  If $\mu_0,\mu_1$ are in $\PdX$, then \eqref{mk} and \eqref{mkdyn} admit a solution.
Let $\Ph$ be a solution of \eqref{mkdyn}. We consider 
\begin{equation}\label{eq-39}
\mu_t:=\Ph(X_t\in\cdot)\in\PdX,\quad t\in\ii
\end{equation}
the time-marginal flow of $\Ph$. The $\PdX$-valued path $[\mu_0,\mu_1]=(\mu_t)_{0\le t\le1}$ is called a \emph{displacement interpolation} between $\mu_0$ and $\mu_1.$ These interpolations encode  \emph{geometric} properties of the manifold $\XX:$ although $\PdX$ is not endowed with a Riemannian metric, Otto \cite{JKO98,Otto01} discovered that $[\mu_0,\mu_1]$ is a minimizing constant speed geodesic path on $\PdX,$ in the length space sense.  In particular, \eqref{eq-42} shows that  for each $x,y,$ $[\delta_x,\delta_y]=(\delta _{\gamma ^{xy}_t})_{0\le t\le1}.$ Therefore, displacement interpolations lift the notion of minimizing constant speed geodesic paths from the state space $\XX$ up to $\PdX$.
\\
The common value of the Monge-Kantorovich problems 
$$
W_2^2(\mu_0,\mu_1)/2:=\inf \eqref{mk}=\inf \eqref{mkdyn}
$$ 
allows to define the Wasserstein distance $W_2(\mu_0,\mu_1)$ between $\mu_0$ and $\mu_1.$ 
Saying that $\gamma ^{xy}$ has a constant speed means that for all $0\le s\le t\le 1,$ $d(\gamma ^{xy}_s,\gamma ^{xy}_t)=(t-s)d(x,y).$ 
With \eqref{eq-42}, we see that $[\mu_0,\mu_1]$ inherits this constant speed property: for all $0\le s\le t\le 1,$ $W_2(\mu_s,\mu_t)=(t-s)W_2(\mu_0,\mu_1).$ 
It is a remarkable fact that $W_2(\mu_0,\mu_1)$ also admits the following Benamou-Brenier representation \cite{BB00}:
\begin{equation}\label{eq-38}
 W_2^2(\mu_0,\mu_1)=
	\inf _{(\nu,v)} \left\{\IiX  |v_t(x)|^2_x\, \nu_t(dx)dt\right\} 
	= \IiX  |\nabla \psi_t(x)|^2_x\, \mu_t(dx)dt
\end{equation}
 where the infimum  is taken over all $(\nu,v)$ such that $\nu=(\nu_t)_{0\le t\le 1}\in C(\ii,\PdX),$ $v$ is a smooth vector field and these quantities are linked by the following current equation (in a weak sense) with boundary values:
\begin{equation*}
\left\{\begin{array}{l}
\partial_t \nu+\nabla\scal(\nu\, v)=0,\quad t\in (0,1)\\
\nu_0=\mu_0,\ \nu_1=\mu_1.
\end{array}\right.
\end{equation*}
The last equality in \eqref{eq-38} states that the infimum is attained at $\nu=\mu$: the displacement interpolation \eqref{eq-39}, and  some    gradient vector field $v=\nabla \psi$ which might not be smooth. The optimal couple $(\mu,\nabla\psi)$ solves the forward-backward coupled system
\begin{equation}\label{eq-41}
(a)\
\left\{\begin{array}{ll}
\partial_t \mu+\nabla\scal(\mu\,\nabla \psi)=0,& t\in (0,1]\\
\mu_0,& t=0
\end{array}\right.
\qquad
(b)\
\left\{\begin{array}{ll}
\partial_t \psi+\ud|\nabla \psi|^2=0, & t\in [0,1)\\
\psi_1,& t=1
\end{array}\right.
\end{equation}
for some measurable function $\psi_1$ which is designed for obtaining $\mu_1$ at time 1.
The  potential $\psi$   is the unique viscosity solution of the Hamilton-Jacobi equation \eqref{eq-41}-(b); it admits the Hopf-Lax representation
\begin{equation}
\psi_t(z)=\inf_y
\left\{\frac{d^2(z,y)}{2(1-t)}+\psi_1(y)\right\} ,\quad 0\le t<1, y\in\XX.\end{equation}
In particular, if $\psi_1$ is bounded, $\psi$ is locally Lipschitz continuous and almost everywhere differentiable.

Based on these properties of the displacement interpolations, F.~Otto \cite{JKO98,Otto01} developed an informal theory aimed at considering the metric space $(\PdX,W_2)$ as a Riemannian manifold. This informal approach relies on the idea that, in view of  the current equation \eqref{eq-41}-(a), $\partial_t \mu _{|t=0}=-\nabla\scal(\mu\,\nabla \psi_0)$ is a candidate to be a tangent vector at $\mu_0$. Second order calculus necessitates to  take also \eqref{eq-41}-(b) into account.

The analogue of displacement interpolation exists with \eqref{mkdyn} replaced by \eqref{sdyn}. This \emph{entropic interpolation} also enjoys properties which are similar to \eqref{eq-38} and \eqref{eq-41}. They are discussed below.

\subsection*{The Monge-Kantorovich problem is a limit of Schr\"odinger problems}

 It is well-known that taking $R^k$ to be the reversible Brownian motion with variance $1/k,$ i.e.\ the Markov measure associated to the Markov generator $$L^k=\Delta/(2k)$$ with the volume measure as its initial distribution,   the bridges of $R^k$ converge: for each $(x,y),$ we have
\begin{equation}\label{eq-54}
\Lim k R ^{k,xy}=\delta _{\gamma ^{xy}} \in\PO
\end{equation}
with respect to the usual narrow topology $\sigma(\PO,C_b(\OO))$. This result is an easy consequence of Schilder's theorem which is a large deviation result (as $k$ tends to infinity)  whose rate function is precisely the dynamic cost function $C$ given at  \eqref{eq-43}, see \cite{DZ}.

In fact, the dynamic and static Monge-Kantorovich problems are respectively the $\Gamma$-limits of sequences of dynamic and static Schr\"odinger problems associated to the sequence $(R^k)_{k\ge1}$ in $\MO$ of reference path measures \cite{Mika04,Leo12a}. More precisely (but still informally), considering the sequence of re-normalized Schr\"odinger problems
\begin{equation}\label{skdyn}
H(P|R^k)/k\to \textrm{min}; \qquad P\in\PO: P_0=\mu_0,P_1=\mu_1,
\tag{S$^k_{\mathrm{dyn}}$}
%\tag{\textsf{s}$^k_{\mathrm{dyn}}$}
\end{equation}
we have 
\begin{equation}\label{eq-50}
\Glim k \eqref{skdyn}=\eqref{mkdyn}
\end{equation}
 and similarly the re-normalized static version
\begin{equation}\label{sk}
H(\pi|R^k _{01})/k\to \textrm{min}; \qquad \pi\in \PXX: \pi_0=\mu_0,\pi_1=\mu_1,
\tag{S$^k$}
%\tag{\textsf{s}$^k$}
\end{equation}
satisfies  $\Glim k \eqref{sk}=\eqref{mk}$.  Recall that this implies that under some compactness requirements, the values converge: $\Lim k \inf\eqref{skdyn}=\inf \eqref{mkdyn}$ and any limit point of the sequence of minimizers $(\Ph^k)_{k\ge1}$ of \eqref{skdyn} solves \eqref{mkdyn}. A similar statement holds with the static problems.
\\
In particular,  the time-marginal flow
\begin{equation*}
\mu^k_t:=\Ph^k_t,\quad t\in\ii, 
\end{equation*}
of the solution to \eqref{skdyn} converges as $k$ tends to infinity to the displacement interpolation $[\mu_0,\mu_1]$ when \eqref{mk} admits a unique solution, for instance when both $\mu_0$ and $\mu_1$ are absolutely continuous.
Denoting and calling the \emph{entropic interpolation} $[\mu_0,\mu_1]^k:=(\mu^k_t)_{t\in\ii}$ of order $k$, we have
\begin{equation}\label{eq-55}
\Lim k [\mu_0,\mu_1]^k=[\mu_0,\mu_1],
\end{equation}
with respect to  the topology of uniform convergence on $C(\ii,\PdX)$ where $\PdX$ is equipped with $W_2$.

\subsection*{The Schr\"odinger problem is a regular approximation of  the Monge-Kantorovich problem}

Now, we explain informally why in some sense, \eqref{skdyn} is a regularization  of its limiting Monge-Kantorovich problem \eqref{mkdyn}. Unlike \eqref{mkdyn}, for each $k\ge1$, \eqref{skdyn} admits a  \emph{unique} solution $\Ph^k\in\PO$. It can be proved that $\Ph^k$ is a Markov diffusion whose semigroup generator $(A^k_t)_{0\le t\le 1}$ is  of the following form 
\begin{equation*}
A^k_t=\nabla \psi^k_t\scal\nabla + \Delta/(2k),\quad 0\le t< 1,
\end{equation*}
with $\psi^k$ the smooth function on $[0,1)\times\XX$ which is the unique classical solution of the Hamilton-Jacobi-Bellman equation
\begin{equation*}
\left\{\begin{array}{ll}
\partial_t \psi^k +\ud	|\nabla \psi^k|^2+\Delta \psi^k/(2k)=0,& t\in [0,1)\\
\psi^k_1,&t=1
\end{array}\right.
\end{equation*}
for some measurable function $\psi^k_1$ designed for recovering\footnote{In fact, one can recover exactly $\mu_1$ if it has a regular density. Otherwise, one can build a sequence $\mu_1^k$ such that $\Lim k \mu_1^k=\mu_1.$} $\mu^k_1=\mu_1$ as the final distribution of the weak solution to
\begin{equation*}
\left\{\begin{array}{ll}
\IiX (\partial_t+A^k_t)u(t,x)\, \mu_t^k(dx)dt=0,& \forall u\in \mathcal{C}_o^\infty((0,1)\ttimes \XX)\\
\mu^k_0=\mu_0,& t=0.
\end{array}\right.
\end{equation*}
which is the evolution equation of the \emph{entropic interpolation} $[\mu_0,\mu_1]^k$ of order $k$.
\\
Remark that the current equation \eqref{eq-41}-(a):
$\left\{\begin{array}{ll}
\partial_t \mu+\nabla\scal(\mu\,\nabla \psi)=0,& t\in (0,1]\\
\mu_0,& t=0
\end{array}\right.$
 signifies
\begin{equation*}
\left\{\begin{array}{ll}
\IiX (\partial_t+A_t)u(t,x)\, \mu_t(dx)dt=0,& \forall u\in \mathcal{C}_o^\infty((0,1)\ttimes \XX)\\
\mu_0,& t=0
\end{array}\right.
\end{equation*}
with
\begin{equation*}
A_t=\nabla \psi_t\scal\nabla, \quad 0\le t< 1,
\end{equation*}
to be compared with the second order  operator $A^k_t$ above. 
We see that, as a consequence of the smoothing and positivity-improving effects of the Laplace operator, the entropic interpolation of order $k$: $[\mu_0,\mu_1]^k,$ is positive and regular on $(0,1)\times\XX$. This is in contrast with the limiting displacement interpolation $[\mu_0,\mu_1].$

\subsubsection*{Extension of the framework}

We have chosen $R^k$ to be  attached to the Brownian motion, but taking $R^k$ to be any Markov measure on a Polish state space $\XX$ satisfying a large deviation principle with some  rate function  $C$ leads to  limiting  Monge-Kantorovich problems associated to alternate cost functions $C$ and $c$ which are still linked by the contraction formula \eqref{eq-46}. Such extensions based on continuous random paths are considered in \cite{Leo12a}. Extensions where the reference measure $R$ is a random walk on a discrete graph are investigated in \cite{Leo12c}, see also Sections \ref{sec-standard} and \ref{sec-slow} below.

\subsection*{New results} Although this article is mainly a survey, we have obtained some new results.
Theorem \ref{res-09} recollects several   sufficient conditions on the reference path measure $R$ and the prescribed marginal measures $\mu_0$ and $\mu_1$, for the unique solution $\Ph$ of \eqref{Sdyn} to admit the following \emph{product}-shaped Radon-Nikodym derivative
\begin{equation}\label{eq-67}
\Ph=f_0(X_0)g_1(X_1)\,R \in\PO,
\end{equation}
where $f_0$ and $g_1$ are positive  measurable positive functions on $\XX.$ The slight innovation is due to the possibility that $R$ might have an infinite mass, e.g.\ the reversible Brownian motion on $\RR^n$. 
Theorem \ref{res-17} is a significant improvement of Theorem \ref{res-09} in the special important case where $R$ is assumed to be Markov. Under some additional requirement on $R$,  it states that \eqref{eq-67} holds where  $f_0$ and $g_1$ may vanish on some sets.
\\
Proposition \ref{res-16} simply states that, if $R$ is Markov, then the solution $\Ph$ of \eqref{Sdyn} is also Markov. Although this is intuitively clear, the author couldn't find in the literature any proof of this  result. Finally,
 the Benamou-Brenier type formulas that are stated at Propositions \ref{res-13} and \ref{res-14}, are new results.

\subsection*{Outline of the paper}

In Section 2, the dynamic and static Schr\"odinger problems are rigorously stated, their main properties of existence and uniqueness are discussed and the shape of their minimizers is described. This specific shape, given by \eqref{eq-67}, suggests to introduce at Section 3 the notion of $(f,g)$-transform of a Markov measure $R$ which is a time-symmetric version of Doob's $h$-transform. In particular, the classical analogue of Born's formula, which was Schr\"odinger's motivation in \cite{Sch31,Sch32}, is derived at Theorem \ref{res-11}.  Then we illustrate at Section 4 the general results of Sections 2 and 3. First, we revisit the case where $R$ is the reversible  Brownian motion. Then, we consider a discrete setting where the reference measure is a reversible random walk on a graph. At Section 5, we see that slowing the reference Markov process down to a complete absence of motion, is the right asymptotic to consider for recovering optimal transport from minimal entropy. Technically, this is expressed in terms of $\Gamma$-convergence results in the spirit of \eqref{eq-50}. In Section 6, by means of basic large deviation results, we present the motivation for addressing the entropy minimization problem \eqref{sdyn}. This leads us naturally to the \emph{lazy gas experiment}, a starting point to the Lott-Sturm-Villani theory. Literature is discussed at Section 7. 

\subsection*{Acknowledgements}
Many thanks to Jean-Claude~Zambrini for numerous fruitful discussions. The author also wishes to thank Toshio~Mikami and a careful referee for pointing out a gap in the preliminary version of the article.

\section{Schr\"odinger's problem}\label{sec-schpb}

We begin  fixing some  notation and describing the general  framework. Then, Schr\"odinger's problem is stated and its main properties are discussed in a general setting.

\subsection*{Path measures}

Depending on the context, we denote by the same letter the set $\OO=C(\ii,\XX)$ of all continuous paths from the unit time interval $\ii$ to the topological state space $\XX,$ or $\OO=D(\ii,\XX)$ the set of all \cadlag\ (right-continuous and left-limited) paths.  We furnish $\XX$ with its Borel $\sigma$-field and $\OO$ with the canonical $\sigma$-field $\sigma(X_t;0\le t\le1)$ which is generated by the time projections
\begin{equation*}
X_t(\omega):= \omega_t\in\XX,\qquad \omega=(\omega_s)_{0\le s\le1}\in\OO,\ t\in\ii.
\end{equation*}
The mapping $X=(X_t)_{0\le t\le1}:\OO\to\OO$ which is the identity on $\OO,$  is usually called the canonical process. We call a \emph{path measure}, any positive measure $Q\in\MO$ on $\OO.$ Its time-marginals are the push-forward measures
\begin{equation*}
Q_t:=(X_t)\pf Q\in \MX,\quad t\in\ii.
\end{equation*}
This means that for any Borel subset $A\subset \XX,$ $Q_t(A)=Q(X_t\in A).$
If $Q$ describes the  behaviour of the random path $(X_t)_{0\le t\le1}$ of some particle, then $Q_t$ describes the behaviour of the random position $X_t$ of the particle at time $t.$ Remark that the flow $(Q_t)_{0\le t\le 1}\in \MX ^{\ii}$ contains less information than the path measure $Q\in\MO.$ In particular, $(Q_t)_{0\le t\le 1}$ doesn't tell us anything about the correlations between two positions at different times $s$ and $t$ which are encoded in $Q _{st}:=(X_s,X_t)\pf Q\in \MXX.$ We shall be primarily concerned with the endpoint marginal measure
\begin{equation*}
Q _{01}:=(X_0,X_1)\pf Q\in\MXX,
\end{equation*}
meaning that for any Borel subsets $B\subset \XXX,$ $Q _{01}(B)=Q((X_0,X_1)\in B).$
We also denote 
$$Q ^{xy}=Q(\cdot\mid X_0=x,X_1=y)\in\PO,$$ the bridge of $Q$ between $x$ and $y$. For each $Q\in\MO,$ the disintegration formula \eqref{eq-09} with $\phi=(X_0,X_1)$ writes as follows:
\begin{equation*}
Q(\cdot)=\IXX Q ^{xy}(\cdot)\,Q _{01}(dxdy)\in\MO.
\end{equation*}
We assume that the topological state space $\XX$ is a Polish (separable and complete metric) space and equip $\OO=D(\ii,\XX)$ with the corresponding Skorokhod topology. It is well-known \cite{Bil68} that this topology turns $\OO$ into a Polish space and that  the corresponding Borel $\sigma$-field is precisely the canonical one: $\sigma(X_t;0\le t\le1).$ Moreover, in restriction to $C(\ii,\XX),$ the Skorokod topology is the topology of uniform convergence which also turns $C(\ii,\XX)$ into a Polish space. We still have the  coincidence of the Borel $\sigma$-field and the  canonical one. 
\\
The path space $\OO$ is furnished with this topology.

\subsection*{Why unbounded path measures}

One may wonder why a random behaviour should be described by an unbounded measure rather than a probability measure. We have in mind as a particular but important application, the \emph{reversible Brownian motion} on $\XX=\Rn.$ It is the Brownian motion whose forward dynamics is driven by the heat semigroup as usual, but its random initial position $X_0$  is uniformly distributed on $\Rn.$  Denoting $R\in\MO$ the corresponding path measure on $\OO=C(\ii,\Rn)$, $R_0(dx)=dx$ is the Lebesgue measure\footnote{Although this paper is not concerned with the interpretation of such a description, one should note that a ``frequencist" interpretation fails unless one introduces an infinite system of independent particles initially distributed according to a Poisson point process with a uniform spatial frequency. An alternate information viewpoint is also relevant: the Lebesgue measure (or any of its positive multiples) is the less informative a priori measure for modelling our complete lack of knowledge about the initial position. Indeed, it is invariant under isometries and translations, and the entropic problems to be considered below are  insensitive to homotheties (up to an additive constant).} on $\Rn$  and $R(\cdot)=\int _{\XX} \mathcal{W}_x(\cdot)\,dx$ where $\mathcal{W}_x$ is the Wiener probability measure with initial marginal $\delta_x$. Clearly, $R$ has the same infinite mass as $R_0.$ 
\\
Similarly, the simple random walk on a countably infinite  graph $\XX$ admits an unbounded reversing measure $R_0$ so that the corresponding reversible simple random walk is described by an unbounded measure $R\in \MO$ with $\OO=D(\ii,\XX).$
\\
Considering such reversible path measures $R\in\MO$ as reference measures usually simplifies  computations.

\subsection*{Relative entropy}

Let $r$ be some $\sigma$-finite positive measure on some  space $Y$. The relative entropy of the probability measure $p$ with respect to $r$ is loosely defined by
\begin{equation}\label{eq-01}
H(p|r):=\int_Y \log(dp/dr)\, dp\in (-\infty,\infty],\qquad p\in \PY
\end{equation}
if $p\ll r$ and $H(p|r)=\infty$ otherwise. 
The rigorous definition of the relative entropy and its basic properties are recalled at the appendix section \ref{sec-entropy}.

\subsection*{Statement of Schr\"odinger's problem}
The main data is a given reference path measure  $R\in\MO.$ In this section any (non-zero) $\sigma$-finite path measure in $\MO$ can serve as a reference measure.
\\
We first state a dynamic version \eqref{Sdyn} of  Schr\"odinger's problem  which is associated   to   $R.$ Then, we define Schr\"odinger's problem \eqref{S} as a static projection of \eqref{Sdyn} and the connections between the solutions of \eqref{S} and \eqref{Sdyn} are described at Proposition \ref{res-03}.

\begin{definition}[Dynamic Schr\"odinger's problem]
The dynamic Schr\"odinger problem associated with the \emph{reference path measure} $R\in\MO$ is the following entropy minimization problem
\begin{equation*}\label{Sdyn}
H(P| R)\to \emph{min};\qquad P\in\PO: P_0=\mu_0,\ P_1=\mu_1
 \tag{S$_{\mathrm{dyn}}$}
\end{equation*}
where $\mu_0, \mu_1\in\PX$ are prescribed initial and final marginals.
\end{definition}
Considering the projection $R _{01}=(X_0,X_1)\pf R\in\MXX$ of $R$ on the product space $\XXX$ as a reference measure, leads us to  Schr\"odinger's  (static) problem.

\begin{definition}[Schr\"odinger's problem]
The (static) Schr\"odinger problem associated with the \emph{reference  measure} $R _{01}\in\MXX$ is the following entropy minimization problem
\begin{equation*}\label{S}
H(\pi| R _{01})\to \emph{min};\qquad \pi\in\PXX: \pi_0=\mu_0,\ \pi_1=\mu_1
 \tag{S}
\end{equation*}
where $\pi_0:=\pi(\cdot\times\XX)$ and $\pi_1:=\pi(\XX\times\cdot)\in\PX$ denote respectively the first and second marginals of $\pi\in\PXX$ and
 $\mu_0, \mu_1\in\PX$ are prescribed marginals.
\end{definition}
These optimization problems are highly connected. This is the content of next proposition. 

\begin{proposition}[F\"ollmer, \cite{Foe85}]\label{res-03}
The Schr\"odinger problems \eqref{Sdyn} and \eqref{S} admit respectively at most one solution $\Ph\in\PO$ and $\ph\in\PXX.$
\\
 If \eqref{Sdyn} admits the solution $\Ph,$ then $\ph=\Ph _{01}$ is the solution of \eqref{S}.
\\
Conversely, if $\ph$ solves \eqref{S}, then  \eqref{Sdyn} admits the solution 
\begin{equation}\label{eq-04}
\Ph(\cdot)=\IXX R ^{xy}(\cdot)\,\ph(dxdy)\in\PO
\end{equation}
which means that $$\Ph _{01}=\ph\in\PXX$$ and that   $\Ph$ shares its bridges with $R:$ 
$$\Ph ^{xy}=R ^{xy},\qquad\forall (x,y)\ \ph\ae$$
\end{proposition}

\begin{proof}
Being  \emph{strictly} convex problems, \eqref{Sdyn} and \eqref{S} admit respectively at most one solution.
\\
Let us  particularize the consequences of the additive  property  formula \eqref{eq-10} to $r=R,$ $p=P$ and $\phi=(X_0,X_1).$  We have for all $P\in\PO,$
\begin{equation*}
H(P|R)=H(P _{01}|R _{01})+\IXX H(P ^{xy}|R ^{xy})\,P _{01}(dxdy)
\end{equation*}
which implies that
$
H(P _{01}|R _{01})\le H(P|R)
$
 with equality (when $H(P|R)<\infty$) if and only if 
$
P ^{xy}=R ^{xy}
$
for $P _{01}$-almost every $(x,y)\in\XXX,$ see \eqref{eq-08} and \eqref{eq-11}.
Note that this additive  property formula  is available since both $\XXX$ and $\OO$ are Polish spaces. Therefore $\Ph$ is the (unique) solution of \eqref{Sdyn} if and only if it disintegrates as \eqref{eq-04}.
\end{proof}

\subsection*{Existence results}
We present below at Proposition \ref{res-04} a simple criterion for \eqref{S} and \eqref{Sdyn} to have a solution. We first need a preliminary result.

\begin{lemma}\label{res-01}
We have: \quad
$\inf \eqref{Sdyn}=\inf \eqref{S}\in (-\infty,\infty].$
\\
Let $B:\XX\to[0,\infty)$ be a measurable function such that 
\begin{equation}\label{eq-61}
\IXX e ^{-B(x)-B(y)}R _{01}(dxdy)<\infty
\end{equation}
and take $\mu_0,\mu_1\in\PX$ such that 
\begin{equation}\label{eq-62}
\IX B\,d \mu_0,\ \IX B\,d \mu_1<\infty.
\end{equation}
The static and dynamic Schr\"odinger problems \eqref{S} and \eqref{Sdyn} admit a (unique) solution if and only if \ $\inf \eqref{Sdyn}=\inf \eqref{S}<\infty$ or equivalently if and only if the prescribed marginals $\mu_0$ and $\mu_1$ are such that
\begin{equation}\label{eq-05}
\textrm{there exists some } \pi^o\in\PXX \textrm{ such that }  \pi_0^o=\mu_0, \pi^o_1=\mu_1 \textrm{ and } H(\pi^o|R _{01})<\infty.
\end{equation}
\end{lemma}

\begin{proof}
The first identity comes from the proof of Proposition \ref{res-03}.
\\
Since $\XX$ is Polish, the probability measures $\mu_0$ and $\mu_1$ are tight measures on $\XX$ and it follows with the Prokhorov criterion on $\XXX$ that the closed constraint set $\Pi(\mu_0,\mu_1):=\left\{\pi\in\PXX: \pi_0=\mu_0,\pi_1=\mu_1\right\}$ is uniformly tight and therefore compact in $\PXX.$ 
\\
Taking \eqref{eq-61} into account,  \eqref{eq-02} and \eqref{eq-03} give us
$	%\begin{equation*}
H(\pi|R _{01})=H(\pi|R _{01}^B)-\IXX W d \pi-z_B,$ $\pi\in\PXX
$	%\end{equation*}
with $W(x,y)=B(x)+B(y),$ $x,y\in\XX$, $z_B:=\IXX e ^{-B(x)-B(y)}R _{01}(dxdy)<\infty$ and $R _{01}^B:=z_B ^{-1}e ^{-B\oplus B}R _{01}\in\PXX.$ In restriction to $\Pi(\mu_0,\mu_1),$ we obtain
\begin{equation}\label{eq-63}
H(\pi|R _{01})=H(\pi|R _{01}^B)-\IX B\, d \mu_0-\IX B\,d \mu_1-z_B,\quad \pi\in\Pi(\mu_0,\mu_1).
\end{equation}
Together with \eqref{eq-62},  this implies that $H(\cdot|R _{01})$ is lower bounded and lower semi-continuous on the compact set $\Pi(\mu_0,\mu_1).$ Hence,  \eqref{S} admits a solution if and only 
$
\inf \eqref{S}<\infty.
$
We already remarked that \eqref{Sdyn} has a solution  if and only \eqref{s} has a solution  that is: $\inf \eqref{S}<\infty$, or equivalently if and only \eqref{eq-05} is satisfied.   
\end{proof}

\begin{proposition}\label{res-04}
Suppose that $R_0=R_1=m\in\MX$ (this is satisfied in particular when $R$ is reversible with $m$ as its reversing measure).
\begin{enumerate}[(a)]
\item
For \eqref{Sdyn} and \eqref{S} to have a solution, it is necessary that $H(\mu_0|m),$ $H(\mu_1|m)<\infty.$
\item
Let us give a set of sufficient conditions.
Suppose that there exist some  nonnegative measurable functions $A$ and $B$ on $\XX$ such that 
\begin{enumerate}
\item[(i)]
$R _{01}(dxdy)\ge e ^{-A(x)-A(y)}\, m(dx)m(dy)$;
\item[(ii)]
$\IXX e ^{-B(x)-B(y)}R _{01}(dxdy)<\infty$
\end{enumerate}
Let $\mu_0$ and $\mu_1$ satisfy
\begin{enumerate}
\item[(iii)]
$\IX (A+B)\,d \mu_0, \IX (A+B)\, d \mu_1<\infty$
\item[(iv)]
$H(\mu_0|m), H(\mu_1|m)<\infty$
\end{enumerate}
Then,  \eqref{Sdyn} and \eqref{S}  admit a unique solution.
\item
If we have (iv) and 
\begin{enumerate}
\item[(v)]
$\IX e ^{\alpha (A+B)}\,dm<\infty$ for some $\alpha>0,$
\end{enumerate}
then (iii) is satisfied.
\end{enumerate}
\end{proposition}

Remark that for (v) to be satisfied, it is necessary that $m$ is a bounded measure.

\begin{proof} Statement (a)  follows directly from Lemma \ref{res-01} and  $H(\mu_0|m),H(\mu_1|m)\le H(\pi^o|R _{01})<\infty,$ see \eqref{eq-08}. 
\\
Let us look at statement (b).
Testing \eqref{eq-05} with  $\pi^o=\mu_0\otimes \mu_1,$ one easily observes that  when (i) and (iii) are satisfied,  it suffices that $H(\mu_0|m), H(\mu_1|m)<\infty$ for \eqref{eq-05} to hold true, and consequently by Lemma \ref{res-01}, for \eqref{Sdyn} and \eqref{S} to have a (unique) solution.
\\ 
Let us look at statement (c). With the variational representation formula \eqref{eq-07}, one sees that (iv) and (v) imply (iii).
\end{proof}

\subsection*{The dual problem}

Take a measurable function $B:\XX\to [1,\infty)$ that satisfies \eqref{eq-61}.  Define $C_B(\XX)$ to be the space of all continuous functions $u:\XX\to\RR$ such that $\sup|u|/B<\infty$ and $\mathrm{P}_B(\XX):=\left\{\mu\in\PX; \IX B\,d \mu<\infty\right\} .$
Based on the variational representation of the relative entropy \eqref{eq-07b} and on the observation that for each $\pi\in\PXX$, $(\pi_0,\pi_1)=(\mu_0,\mu_1)$ if and only if $\IXX [\varphi(x)+\psi(y)]\,\pi(dxdy)=\IX \varphi\,d \mu_0+\IX \psi\, d \mu_1,$ for all $\varphi, \psi\in C_b(\XX)$ (the space of all bounded continuous numerical functions on $\XX$), it can be proved  that a dual problem to the Schr\"odinger problem \eqref{S} is
\begin{equation}\label{D}
\IX \varphi\,d\mu_0+\IX \psi\, d\mu_1-\log\IXX e ^{\varphi\oplus\psi}\,dR _{01}\to \textrm{max};\qquad \varphi, \psi\in C_B(\XX)
\tag{D}
\end{equation}
where $\varphi\oplus \psi: (x,y)\in\XXX\mapsto \varphi(x)+\psi(y)\in\RR$ and it is assumed that the prescribed marginals satisfy $\mu_0,\mu_1\in\mathrm{P}_B(\XX).$ 
\\
In particular, the dual equality $\inf \eqref{S}=\sup \eqref{D}\in (-\infty,\infty]$ is satisfied.  This is proved, for instance, in \cite{Leo01c} when the reference measure is a probability measure. In the general case, take \eqref{eq-63} into account to get back to a  reference measure with a finite mass. Of course, there is no reason for the dual attainment to hold in general in a space of regular functions such as $C_B(\XX)^2.$ Suppose however that $\mu_0$ and $\mu_1$ are such that \eqref{D} is attained at $(\hat\varphi,\hat\psi).$ Then, the dual equality: $\IXX \hat\varphi\oplus \hat\psi\, d\ph-\log\IXX e ^{\hat\varphi\oplus\hat\psi}\,dR _{01}=H(\ph|R _{01})$ and the case of equality in \eqref{eq-07b} lead us to
\begin{equation}\label{eq-25}
d\ph/dR _{01}=\exp(\hat\varphi\oplus \hat\psi),
\end{equation}
at least when $d\ph/dR _{01}>0$.
The shape of the minimizer $\ph$ of \eqref{S} will be discussed further in next subsection.
\\
Similarly, a dual problem to the dynamic version \eqref{Sdyn} of \eqref{S} is
\begin{equation}\label{Ddyn}
\IX \varphi\,d\mu_0+\IX \psi\, d\mu_1-\log\IO e ^{\varphi(X_0)+\psi(X_1)}\,dR \to \textrm{max};\qquad \varphi, \psi\in C_B(\XX)
\tag{D$_{\textrm{dyn}}$}
\end{equation}
We observe that \eqref{D}=\eqref{Ddyn}.

\subsection*{Some properties of the minimizer $\ph$  of \eqref{S}}

 We give some details about the  structure of the unique minimizer $\ph$ of \eqref{S} which is assumed to exist; for instance under the hypotheses of Proposition \ref{res-04}. 

It is proved in \cite[Thm.\,5.1 \& (5.9)] {Leo01b}\footnote{
The assumptions of \cite{Leo01b} require that $R _{01}$ is a probability measure. In the general case where $R _{01}$ is unbounded, use \eqref{eq-63} to go back to the unit mass setting.} that there exist two functions $\varphi, \psi:\XX\to\RR$ such that
\begin{enumerate}
\item[(i)]
$d\ph/d R _{01}=\exp(\varphi\oplus \psi), \ \ph\ae$
\item[(ii)]
 $\varphi\oplus \psi: \XXX\to\RR$ is $R _{01}$-measurable.
\end{enumerate}
It is  tempting to write that $\ph$ has the following shape
\begin{equation*}
\ph(dxdy)=f(x)g(y)\, R _{01}(dxdy)
\end{equation*}
where $f=e ^{\varphi}$ and $g=e ^\psi$ are such that the marginal constraints
\begin{equation}\label{eq-12}
\left\{\begin{array}{lcll}
f(x) E_R[g(X_1)\mid X_0=x]&=&d \mu_0/d R_0 (x),&\  R_0\ae\\
g(y) E_R[f(X_0)\mid X_1=y]&=&d \mu_1/dR_1 (y),&\  R_1\ae
\end{array}\right.
\end{equation}
are satisfied. This was already suggested by \eqref{eq-25}. But this is not allowed in the general case. Indeed, two obstacles have to be avoided.
Some comments are necessary. 

\subsubsection*{Obstacle (i)}
Firstly, the identity (i) is only valid $\ph$-almost everywhere and it might happen that it doesn't hold true $R _{01}$-almost everywhere. Otherwise stated, there exists some measurable subset $S \subset \XXX$ such that:
\begin{enumerate}
\item[(i)']
$\ph=\1 _{S}\exp(\varphi\oplus \psi)\, R _{01}$
\end{enumerate}
and it is not true in general that the set $S\subset\XXX$ is the product $S=S_0\times S_1$ of two measurable subsets of $\XX,$ see \cite[\S 2]{FG97} or \cite[\S 5]{Leo01b}.

\subsubsection*{Obstacle (ii)}
Secondly, statement (ii) does not imply that $\varphi$ and $\psi$ are respectively $R_0$ and $R_1$-measurable  on $\XX$. Only the tensor sum $\varphi\oplus \psi$ is $R _{01}$-measurable on the product space $\XXX.$ Hence, one is not allowed to consider the conditional expectations in \eqref{eq-12}.

\subsubsection*{Avoiding obstacle (i)}

To  avoid the obstacle (i), it is enough to slightly modify the prescribed marginals $\mu_0$ and $\mu_1$ as follows. It is shown in \cite{Leo01b} that  (i) is satisfied $R _{01}\ae$ (rather than $\ph\ae$) if $(\mu_0,\mu_1)$ is in the intrinsic core: $\icor \mathcal{C},$  of the set of all admissible constraints $$\mathcal{C}:=\{(\mu_0,\mu_1)\in\PX^2; \inf\eqref{S}_{(\mu_0,\mu_1)}<\infty\}.$$ Recall that for any convex set $C$, its intrinsic core is defined as\\ $\icor C:=\left\{y\in C; \exists x,z\in C: y\in ]x,z[\right\}$ where $]x,z[:=\left\{(1-t)x+tz; 0<t<1\right\}\subset C.$ It is also shown in \cite{Leo01b} that $\mathcal{C}=\{ \Lambda^*<\infty\}$ where $\Lambda^*$ is the convex conjugate of the extended real valued function $\Lambda$ which is defined for any  measurable functions $\varphi,\psi$  by $\Lambda(\varphi,\psi)=\log \IXX e ^{\varphi\oplus \psi}\,dR _{01}\in(-\infty,\infty].$ Therefore $\mathcal{C}$  is a convex subset of $\PX ^2.$
 In particular, considering 
\begin{equation}\label{eq-60}
\left\{\begin{array}{lcl}
\mu_0^\epsilon&:=&(1-\epsilon)\mu_0+\epsilon R_0^w\\
\mu_1^\epsilon&:=&(1-\epsilon)\mu_1+\epsilon R_1^w
\end{array}\right.
\end{equation}
with  $R_0^w, R_1^w\in\PX$ the marginals of $R^w _{01}=z_w ^{-1}e ^{-w}\,R _{01}\in\PXX$  where the function $w$ is  chosen\footnote{In the important special case where $R$ is a probability measure, just take $w=0$ in order that $R^w _{01}=R _{01}.$} such that $\IXX w e ^{-w}\,dR _{01}<\infty$ for $H(R^w _{01}|R _{01})<\infty$ to be satisfied with $\IXX w\, dR^w _{01}<\infty,$ see \eqref{eq-03}, we observe that for any admissible $(\mu_0,\mu_1)\in \mathcal{C},$ for any  arbitrarily small $\epsilon>0$, $(\mu_0^\epsilon,\mu_1^\epsilon)\in\icor \mathcal{C}$.  Therefore, $(\mu_0^\epsilon,\mu_1^\epsilon)$ is arbitrarily close  to $(\mu_0,\mu_1)$ in total variation norm and the corresponding solution $\ph ^\epsilon$ of $\eqref{S}_{(\mu_0^\epsilon,\mu_1^\epsilon)}$ satisfies
\begin{equation*}
\ph^\epsilon=\exp (\varphi^\epsilon\oplus \psi^\epsilon)\, R _{01}
\end{equation*}
for some functions $\varphi^\epsilon$ and $\psi^\epsilon$ such that $\varphi^\epsilon\oplus \psi^\epsilon$ is jointly $R _{01}$-measurable.

\begin{proposition}\label{res-05}
We say that the constraint $(\mu_0,\mu_1)$ is \emph{internal} if it is in the intrinsic core of the set of all admissible constraints: $(\mu_0,\mu_1)\in\icor \mathcal{C}$. In this case, we have 
\begin{equation*}
\ph=\exp(\varphi\oplus \psi)\, R _{01} 
\end{equation*}
for some jointly $R _{01}$-measurable function $\varphi\oplus \psi$ on $\XXX.$
\end{proposition}

\subsubsection*{Overcoming obstacle (ii)}

To overcome the measurability obstacle (ii), it is necessary to impose some restriction on the reference measure $R _{01}.$   It is proved in \cite[Proposition\,6.1]{BL92} that when the function $\varphi\oplus \psi$ is measurable with respect to some product measure $\alpha\otimes \beta$ on the product space $\XXX,$ the functions $\varphi$ and $\psi$ are respectively $\alpha$-measurable and $\beta$-measurable. Therefore, if it is assumed that $$R_0\otimes R_1\ll R _{01},$$
any $R _{01}$-measurable function is $R_0\otimes R_1$-measurable. As $\varphi\oplus \psi$ is $R _{01}$-measurable, $\varphi$ and $\psi$ are respectively $R_0$ and $R_1$-measurable. 
 \\
Choosing  another way,
it is also possible to obtain the desired measurability property when assuming 
$$
R _{01}\ll R_0\otimes R_1.
$$
To see this, let $\ph,$ $\varphi$ and $\psi$ satisfy (i).
It is proved in \cite{Csi75} that $\varphi\oplus \psi$ stands in  the $L^1(\ph)$-closure  of $L^1(\mu_0)\oplus L^1(\mu_1),$ denoted by $\Lambda _{\ph}$. In addition, with \cite[Proposition 2]{RT93}, we know that when $\pi\in \PXX$ with first and second marginals $\pi_0=\mu_0$ and $\pi_1=\mu_1,$  is such that
$\pi\ll\pi_0\otimes \pi_1=\mu_0\otimes \mu_1$, then for any $\theta\in \Lambda_\pi,$ there exist two \emph{measurable} functions $\theta_0$ and $\theta_1$ such that $\theta=\theta_0\oplus \theta_1,$ $\pi\as$ 
Therefore, if 
\begin{equation*}
\ph\ll \mu_0\otimes \mu_1,
\end{equation*}
then $\varphi$ and $\psi$ are respectively $\mu_0$ and $\mu_1$-measurable functions.
\\
But this implicit criterion is not easy to check. The following result is more practical.

\begin{proposition}\cite[Thm.\,3]{RT93}\label{res-06}
If $R _{01}\ll R_0\otimes R_1,$ then $\varphi$ and $\psi$ in (i): $d\ph/d R _{01}=\exp(\varphi\oplus \psi), \ \ph\ae$, can be chosen  respectively as $\mu_0$ and $\mu_1$-measurable functions.
\end{proposition}

\begin{proof}
By assumption $\ph\ll R _{01}\ll R_0\otimes R_1$. But $\ph\ll R_0\otimes R_1$  implies that $\ph\ll \ph_0\otimes\ph_1=\mu_0\otimes \mu_1$  and we conclude as above.
\end{proof}

Extending the functions $\varphi$ and $\psi$ to their $R_0$ and $R_1$-measurable versions: $\1 _{\{d \mu_0/dR_0>0\}}\varphi$ and $\1 _{\left\{d \mu_1/dR_1>0\right\}} \psi$,  we see that  $\varphi$ and $\psi$ can be taken respectively $R_0$ and $R_1$-measurable.

\subsubsection*{Summing up}

Putting Propositions \ref{res-04}, \ref{res-05} and the above considerations  together with \eqref{eq-12},  we obtain the following

\begin{theorem}\label{res-08}
Suppose that $R$ satisfies
\begin{enumerate}
\item[(i)]
$R_0=R_1=m\in\MX$;
\item[(ii)]
$R _{01}(dxdy)\ge e ^{-A(x)-A(y)}\, m(dx)m(dy)$ for some nonnegative measurable function $A$ on $\XX;$
\item[(iii)]
$\IXX e ^{-B(x)-B(y)}R _{01}(dxdy)<\infty$ for some nonnegative measurable function $B$ on $\XX;$
\item[(iv)]
$m ^{\otimes 2}\ll R _{01}$ or $R _{01}\ll m ^{\otimes 2}.$
\end{enumerate}
Suppose also that the constraint $(\mu_0,\mu_1)$ satisfies
\begin{enumerate}
\item[(v)]
$H(\mu_0|m), H(\mu_1|m)<\infty;$
\item[(vi)]
$\IX (A+B)\,d \mu_0, \IX (A+B)\, d \mu_1<\infty$ where $A$ and $B$  appear at (ii) and (iii) above;
\item[(vii)]
$(\mu_0,\mu_1)$ is internal, see the statement of Proposition \ref{res-05}.\\ This is the case for instance when $m$ is a probability measure and $\mu_0,\mu_1\ge \epsilon m$, for some $\epsilon>0.$
\end{enumerate}
Then, \eqref{S} admits a unique solution $\ph$ and 
\begin{equation}\label{eq-14}
\ph(dxdy)=f_0(x)g_1(y)\, R _{01}(dxdy)
\end{equation}
where the positive  functions $f_0$ and $g_1$  are  $m$-measurable and solve :
\begin{equation}\label{eq-13}
\left\{\begin{array}{lcll}
f_0(x) E_R[g_1(X_1)\mid X_0=x]&=&d \mu_0/dm (x),&\  m\ae\\
g_1(y) E_R[f_0(X_0)\mid X_1=y]&=&d \mu_1/dm (y),&\  m\ae
\end{array}\right.
\end{equation}
which is called the \emph{Schr\"odinger system}\footnote{The article \cite{RT93} refers to \eqref{eq-13} as the Schr\"odinger equation, but this is  misleading. After Fortet and Beurling \cite{Fort40,Beu60}, we prefer  calling \eqref{eq-13} the Schr\"odinger system.}.
\end{theorem}

\begin{itemize}
\item
It is not necessary for $E_R[g_1(X_1)\mid X_0]$ and $E_R[f_0(X_0)\mid X_1]$
 to be well-defined that $f_0(X_0)$ and $g_1(X_1)$ are $R$-integrable, since $f_0$ and $g_1$ are positive measurable functions. Only a notion of  integration of \emph{nonnegative} functions is required, see \cite{Leo12b}.

\item
The assumption (vii) is here to make sure that $d\ph/dR _{01}>0.$ If it is not satisfied,  $d\ph/dR _{01}$ may not have a product form and its structure may be quite complex. The complete description of $d\ph/dR _{01}$  in this case is given in \cite{Leo01b}.

\item
In view of \eqref{eq-60}, for the assumption (vii) to hold, it is enough that $
\left\{\begin{array}{lcl}
\mu_0&\ge& \epsilon R^w_0\\ \mu_1&\ge &\epsilon R^w_1
\end{array}\right.
,$ for some $\epsilon>0,$ and we can choose $w=0$ when $m$ is a probability measure.
\end{itemize}

\subsection*{The solution $\Ph$ of \eqref{Sdyn}}
We  deduce from this theorem the characterization of $\Ph.$

\begin{theorem}\label{res-09}
Suppose that the hypotheses of Theorem \ref{res-08}  are satisfied.
Then, \eqref{Sdyn} admits the unique solution
\begin{equation}\label{eq-15}
\Ph=f_0(X_0)g_1(X_1)\,R \in\PO
\end{equation}
where $f_0$ and $g_1$ are the measurable positive functions which appear at \eqref{eq-14} and solve \eqref{eq-13}.
\end{theorem}

\begin{proof}
The existence of the solution $\Ph$ and its representation by the Radon-Nikodym formula \eqref{eq-15} are direct consequences of Proposition \ref{res-03},  Theorem \ref{res-08}, \eqref{eq-04} and \eqref{eq-14}.
\end{proof}

\subsection*{The special case where $R$ is Markov}

We are going to  assume that the reference path measure $R$ is  Markov. Under this restriction, we obtain at Theorem \ref{res-17} below a more efficient version of Theorem \ref{res-08}. Let us recall the time-symmetric definition of the Markov property.

\subsubsection*{Markov property}

One says that $R\in\MO$ is \emph{Markov} if its time marginal $R_0$ (or any other time marginal $R _{t_o}$ with $0\le t_o\le 1$) is $\sigma$-finite\footnote{This assumption is necessary for defining conditional versions of $R$, such as $R(\cdot\mid X_t)$ or $R(\cdot\mid X _{[0,t]})$, see \cite{Leo12b}.} and for each $0\le t\le 1,$
\begin{equation*}
R(X _{[0,t]}\in \cdot, X _{[t,1]}\in \cdot\cdot\mid X_t)
=R(X _{[0,t]}\in \cdot\mid X_t)R( X _{[t,1]}\in \cdot\cdot\mid X_t)
\end{equation*}
signifying that under $R$, for any $t$, conditionally on the present state $X_t$ at time $t$, past and future are independent. This is equivalent to the usual forward time-oriented  Markov property
\begin{equation*}
R(X _{[t,1]}\in\cdot|X _{[0,t]})=R(X _{[t,1]}\in\cdot|X _{t}),\quad \forall t\in\ii.
\end{equation*}

\begin{proposition}\label{res-16}
Suppose that the reference measure $R\in\MO$ is Markov.
If it exists, the solution $\Ph$ of \eqref{Sdyn} is also Markov.
\end{proposition}

\begin{proof}
We need some notation. For each $0\le t\le 1,$ we set $\OO _{[0,t]}:=\left\{\omega _{|[0,t]}; \omega\in\OO\right\} $ and $\OO _{[t,1]}:=\left\{\omega _{|[t,1]};\omega\in\OO\right\} $ the set of all paths on $[0,t]$ and $[t,1]$ respectively. For any $Q\in\MO,$ $Q ^{t,z}:=Q(\cdot|X_t=z)\in\PO,$ $Q ^{t,z}_{[0,t]}:=Q(X _{[0,t]}\in\cdot|X_t=z)\in \mathrm{P}(\OO _{[0,t]})$ and 
$Q ^{t,z}_{[t,1]}:=Q(X _{[t,1]}\in\cdot|X_t=z)\in \mathrm{P}(\OO _{[t,1]}).$ 

\begin{claim}
We  fix $0\le t\le 1$. Among all the $P\in\PO$ such that $P_t=\mu,$ $P ^{tz}_{[0,t]}=Q_<^{tz}$ and $P ^{tz}_{[t,1]}=Q_>^{tz}$, $z\in\XX,$ where $\mu\in\PX,$ $Q_<^{tz}\in \mathrm{P}(\OO _{[0,t]}\cap \left\{X_t=z\right\} )$
and $Q_>^{tz}\in \mathrm{P}(\OO _{[t,1]}\cap \left\{X_t=z\right\} )$ are prescribed, the relative entropy $H(\cdot|R)$ attains its unique minimum  at $P^*(\cdot)=\IX Q_<^{tz}\otimes Q_>^{tz}(\cdot)\,\mu(dz).$ In particular, $P ^{*}_{[t,1]}(\cdot|X _{[0,t]})=P ^{*}_{[t,1]}(\cdot|X_t).$ 
\end{claim}
Accept this claim for a while and suppose, ad absurdum, that $\Ph$ is not Markov. Then, there exists some $0\le t\le 1$ such that $\Ph(\cdot|X_t)\not=\Ph_{[0,t]}(\cdot|X_t)\otimes \Ph_{[t,1]}(\cdot|X_t).$ Choosing $\mu=\Ph_t,$ $Q_<^{tz}=\Ph ^{tz}_{[0,t]}$ and $Q_>^{tz}=\Ph ^{tz}_{[t,1]}$ in the above claim, we see that the time marginals are unchanged: $P^*_s=\Ph_s,$ for all $s\in\ii$ and that $H(P^*|R)<H(\Ph|R):$   $\Ph$ is not the solution to \eqref{Sdyn}, a contradiction.

It remains to prove the claim. With \eqref{eq-10}, we see that
\begin{equation*}
H(P|R)=H(\mu|R_t)+\IX H(P ^{tz}|R ^{tz})\, \mu(dz)
\end{equation*}
and 
\begin{equation*}
H(P ^{tz}|R ^{tz})=H(Q_<^{tz}|R ^{tz}_{[0,t]})+\int _{\OO _{[0,t]}}H(P ^{tz}_{[t,1]}(\cdot|X _{[0,t]})|R ^{tz}_{[t,1]}(\cdot|X_{[0,t]}))\,dQ_<^{tz}.
\end{equation*}
Noting that $R ^{tz}_{[t,1]}(\cdot|X_{[0,t]})=R ^{tz}_{[t,1]}$ and  $Q_> ^{tz}=\int _{\OO _{[0,t]}} P ^{tz}_{[t,1]}(\cdot|X _{[0,t]})\,dQ_<^{tz}$, we obtain with Jensen's inequality that 
\begin{eqnarray*}
H\big(Q_> ^{tz}\big|R ^{tz}_{[t,1]}\big)&=&H \Big(\int _{\OO _{[0,t]}} P ^{tz}_{[t,1]}(\cdot|X _{[0,t]})\,dQ_<^{tz} \Big|R ^{tz}_{[t,1]}(\cdot|X_{[0,t]})\Big)\\
&\le&
	\int _{\OO _{[0,t]}} H ( P ^{tz}_{[t,1]}(\cdot|X _{[0,t]})|R ^{tz}_{[t,1]}(\cdot|X _{[0,t]}))\,dQ_<^{tz} 
\end{eqnarray*}
with equality if and only if $P ^{tz}_{[t,1]}(\cdot|X _{[0,t]})=Q ^{tz}_>,$\quad $Q ^{tz}_<\ae$ Since this holds for $\mu$-almost every $z,$  this amounts to say that $P(\cdot)=\IX Q_<^{tz}\otimes Q_>^{tz}(\cdot)\,P_t(dz)$ and it also means that $P ^{tz}_{[t,1]}(\cdot|X _{[0,t]})=P ^{tz}_{[t,1]},$ $P ^{tz}_{[0,t]}\ae$ which is the desired forward Markov property at time $t$. This completes the proofs of the claim and the proposition.
\end{proof}

\subsubsection*{Reversibility}

A path measure $R\in\MO$ is said to be \emph{reversible} with $m\in\MX$
as its reversing measure ($m$-reversible for short), if $R_0=m$ and for any $0\le u\le v\le1,$ $R$ is invariant with respect to the time reversal mapping $\mathrm{rev} ^{uv}$ defined by: $\mathrm{rev} ^{uv}_t:=X _{(u+v-t)^+}, u\le t\le v$, meaning that $(\mathrm{rev} ^{uv})\pf R=(X _{[u,v]})\pf R.$ 

Clearly, this implies that $R_u=R_v$ for any $u,v$. In other words, $R$ is $m$-\emph{stationary} i.e.\ $R_t=m,$ for all $0\le t\le 1.$

This notion is invoked at statement (c) of the following result.

\begin{theorem}[The Markov case]\label{res-17}\ 

Suppose that the reference measure $R\in\MO$  satisfies
\begin{enumerate}
\item[(i)]
$R$ is Markov;
\item[(ii)]
there exist some $0<t_o<1$ and some measurable $\XX_o\subset\XX$ such that $R _{t_o}(\XX_o)>0$ and 
\begin{equation*}
R _{01}\ll R\big((X_0,X_1)\in\cdot|X _{t_o}=z\big),\quad \forall z\in\XX_o.
\end{equation*}
\item[(iii)]
$R_0=R_1=m\in\MX$;
\item[(iv)]
$R _{01}(dxdy)\ge e ^{-A(x)-A(y)}\, m(dx)m(dy)$ for some nonnegative measurable function $A$ on $\XX;$
\item[(v)]
$\IXX e ^{-B(x)-B(y)}R _{01}(dxdy)<\infty$ for some nonnegative measurable function $B$ on $\XX;$
\end{enumerate}
Suppose also that the constraint $(\mu_0,\mu_1)$ satisfies
\begin{enumerate}
\item[(vi)]
$H(\mu_0|m), H(\mu_1|m)<\infty;$
\item[(vii)]
$\IX (A+B)\,d \mu_0, \IX (A+B)\, d \mu_1<\infty$ where $A$ and $B$  appear at (iv) and (v) above.
\end{enumerate}
\begin{enumerate}[(a)]
\item
Then,  the unique solution $\Ph$  of \eqref{Sdyn} is also Markov and
\begin{equation}\label{eq-15b}
\Ph=f_0(X_0)g_1(X_1)\,R \in\PO
\end{equation}
where $f_0$ and $g_1$ are the $m$-measurable \emph{nonnegative} functions which appear at \eqref{eq-14} and solve the Schr\"odinger system \eqref{eq-13}.
\item
Conversely, let $\Ph$ be defined by \eqref{eq-15b} with $f_0$ and $g_1$ two $m$-measurable nonnegative functions solving the Schr\"odinger system \eqref{eq-13}. Then, $\Ph$ is Markov and it is the unique solution of \eqref{Sdyn}.

\item
For the properties (i), (ii) and (iii) to hold, it is enough that $R$ is a $m$-reversible Markov measure which admits a regenerative set in the following sense: There exists a measurable subset $\XX_o\subset\XX$ with $m(\XX_o)>0$ such that for each $0<h<1$ and all measurable subset $A\subset\XX$ with $m(A)>0,$ we have:
$	%\begin{equation*}
R(X_h\in A|X_0=x)>0,$ for all $x\in\XX_o.
$	%\end{equation*}
\end{enumerate}
\end{theorem}

\begin{proof} \boulette{(a)}
By Proposition \ref{res-04}, the properties (iii)--(vii) assure the existence of the unique solution $\Ph$ of \eqref{Sdyn}.
\\
With Proposition \ref{res-03}, we have $\Ph ^{xy}=R ^{xy},$ for all $(x,y),$ $\ph\ae$ This means that 
\begin{equation*}
\frac{d\Ph}{dR}=\frac{d\ph}{dR _{01}}(X_0,X_1).
\end{equation*}
On the other hand, we have just seen at Proposition \ref{res-16} that $\Ph$ is Markov. But, it is proved in \cite{LRZ12} that under the assumptions (i) and (ii), if $P\in\PO$ is a Markov measure such that $dP/dR=h(X_0,X_1)$ for some measurable function $h,$ then there exist two measurable nonnegative functions $f$ and $g$ such that $P=f(X_0)g(X_1)\,R$. This proves statement (a).
\Boulette{(b)}
The fact that $\Ph$ is the solution of \eqref{Sdyn} is proved in \cite{Csi75} by a geometric approach or in \cite{Leo01b} by a functional analytic approach. We easily see that $\Ph$ inherits the Markov property of $R$ using the time-symmetric definition of the Markov property together with the product shape of \eqref{eq-15b}.
\Boulette{(c)}
Statement (c) is an easy exercise.
\end{proof}

Remark that, unlike Theorem \ref{res-08},  Theorem \ref{res-17} does not require that the constraint $(\mu_0,\mu_1)$ is internal. Also remark that the functions $f_0$ and $g_1$ are nonnegative (in contrast with Theorem \ref{res-08} where they are positive) and that it may happen that $R_0(f_0=0)$ or $R_1(g_1=0)$ is positive. 

Theorem \ref{res-17} extends a similar result by F\"ollmer and Gantert in \cite{FG97}, where it is required  for the product shape formula \eqref{eq-15b} to hold, that $R\ll\Ph$ and also that $\XX_o$ has full measure.

\section{$(f,g)$-transform of a Markov measure}

Motivated by Theorems \ref{res-09} and \ref{res-17}, we introduce the transform $f_0(X_0)g_1(X_1)\,R$ of a Markov measure $R$ and call it an $(f,g)$-transform. It was already noticed by F\"ollmer \cite{Foe85,FG97} and Nagasawa \cite{Naga89}, that it is a time symmetric version of Doob's usual $h$-transform \cite{Doob57,Doob84}.

We are going to  assume for simplicity that the reference path measure $R$ is  reversible. Let us recall the definition of this notion.

\subsection*{$(f,g)$-transform of a reversible Markov measure}
Let us first state an assumption which will hold for the remainder of the paper.

\begin{assumption}
The reference path measure $R\in\MO$ is  Markov and $m$-reversible with $m\in\MX$. 
\end{assumption}
The Markov property of the reference measure will turn out to be crucial for the description of the dynamics of the solution $\Ph$ of \eqref{Sdyn}.  Indeed, we have already seen at Proposition \ref{res-16} that $\Ph$ inherits the Markov property from $R.$ It follows that its dynamics is characterized by its stochastic derivatives. On the other hand,  reversibility  is only assumed for simplicity. 

\begin{definition}[$(f,g)$-transform]\label{def-01}
%The reference  measure $R\in\MO$ is assumed to be Markov.
Let $f_0$, $g_1:\XX\to[0,\infty)$ be two nonnegative measurable functions such that $E_R(f_0(X_0)g_1(X_1))=1$. The path measure
\begin{equation}\label{eq-17}
P:=f_0(X_0)g_1(X_1)\, R\in\PO
\end{equation}
is called an $(f,g)$-transform of $R.$
\end{definition}

This definition is motivated by Theorems \ref{res-09} and \ref{res-17}  which assert that the solution $\Ph$ of \eqref{Sdyn} is an $(f,g)$-transform of $R$. Note that under Theorem \ref{res-09}'s assumptions, $f_0$ and $g_1$ are positive, while they are allowed to vanish under Theorem \ref{res-17}'s assumptions, as in Definition \ref{def-01}.

Let us introduce for each $t\in\ii,$ the functions $f_t,g_t:\XX\to[0,\infty)$  defined by
\begin{equation}\label{eq-20}
\left\{\begin{array}{lcl}
f_t(z)&:=& E_R(f_0(X_0)\mid X_t=z)\\
g_t(z)&:=& E_R(g_1(X_1)\mid X_t=z)
\end{array}\right.,\quad \textrm{for } P_t\ae\ z\in\XX.
\end{equation}
Remark that although we have $E_R(f_0(X_0)g_1(X_1))<\infty,$ this does not ensure that $f_0(X_0)$ and $g_1(X_1)$ are integrable. We have to use positive integration to give a meaning to the conditional expectations  $E_R(f_0(X_0)\mid X_t),E_R(g_1(X_1)\mid X_t)\in[0,\infty]$, see  \cite{Leo12b}.

Next result is a kind of converse of Theorems \ref{res-09} and \ref{res-17}.

\begin{theorem}\label{res-12}
If the functions $f_0$ and $g_1$ entering the definition of the $(f,g)$-transform $P$ of $R$ given at \eqref{eq-17} satisfy 
\begin{equation*}
\left\{\begin{array}{l}
\IX g_0 f_0\log f_0 \,dm<\infty\\
\IX f_1g_1\log g_1\,dm<\infty
\end{array}\right.
\end{equation*}
(as a convention $0\log 0=0$), then $P _{01}$ and $P$ are the unique solutions of \eqref{S} and \eqref{Sdyn} respectively, where the prescribed constraints $\mu_0,$ $\mu_1\in\PX$ are chosen to satisfy \eqref{eq-13}, i.e.\ using notation \eqref{eq-20}
\begin{equation}\label{eq-21}
\left\{\begin{array}{l}
\mu_0= f_0g_0\,m\\
\mu_1= f_1g_1\,m
\end{array}\right..
\end{equation}
\end{theorem}

\begin{proof}
The pair of finite integrals is equivalent to $H(P|R)<\infty.$ Hence, the statement about \eqref{S} is a direct consequence of \cite[Thm. 5.1]{Leo01b}. Its corollary about $\eqref{Sdyn}$ follows as for Theorem \ref{res-09}.
\end{proof}
Note that there may exist solutions of \eqref{Sdyn} which are not $(f,g)$-transforms of $R$. This happens when the support of the solution is not  a rectangle (i.e.\,the product of Borel subsets), see \cite[\S2]{FG97} or \cite[\S5]{Leo01b}.
\\
Next result extends the product formulas \eqref{eq-21} to all $t\in\ii.$

\begin{theorem}[Euclidean analogue of Born's formula]\label{res-11}
The path measure $P=f_0(X_0)g_1(X_1)\, R$ is Markov and for each $0\le t\le1,$ its time marginal $P_t\in\PX$ is given by
\begin{equation}\label{eq-16}
P_t=f_tg_t \,m.
\end{equation}
\end{theorem}

\begin{remark}
It follows with \eqref{eq-16} that for all $t\in\ii,$
$
0<E_R(f_0(X_0)\mid X_t),\ E_R(g_1(X_1)\mid X_t)<\infty,$ $P\ae,$
but not $R\ae$ in general.
\end{remark}

\begin{proof}[Proof of Theorem \ref{res-11}]
The Markov property of $P$ is a direct consequence of Theorem \ref{res-12} and Proposition \ref{res-16}. 
\\
We propose an alternate simple proof.
To prove  that $P$ is Markov, we show that for each $0\le t\le 1$ and  any bounded measurable functions $a\in \sigma(X _{[0,t]})$, $b\in \sigma(X _{[t,1]})$, we have
\begin{equation*}
E _{P}(ab\mid X_t)=E _{P}(a\mid X_t) E _{P}(b\mid X_t).
\end{equation*}
Indeed,   we have
\begin{eqnarray*}
E _{P}(ab\mid X_t)&\overset{\textrm{(i)}}{=}& \frac{E_R(f_0(X_0)abg_1(X_1)\mid X_t)}{E_R(f_0(X_0)g_1(X_1)\mid X_t)} 
	\overset{\textrm{(ii)}}{=} \frac{E_R(f_0(X_0)a\mid X_t)E_R(bg_1(X_1)\mid X_t)}{E_R(f_0(X_0)\mid X_t)E_R(g_1(X_1)\mid X_t)} \\
&\overset{\textrm{(iii)}}{=}& E _{P}(a\mid X_t)E _{P}(b\mid X_t),\quad P\ae,
\end{eqnarray*}
which is the desired result. Equality (i) is a general result about conditioning; note that we do not divide by zero $P\as$ Equality (ii) uses crucially the assumed Markov property of $R$ and one obtains (iii) by considering separately the cases when $b\equiv 1$ and $a\equiv 1$ in the just obtained identity to see that $E _{P}(a\mid X_t)=\frac{E_R(f_0(X_0)a\mid X_t)}{E_R(f_0(X_0)\mid X_t)}$ and $E _{P}(b\mid X_t)=\frac{E_R(bg_1(X_1)\mid X_t)}{E_R(g_1(X_1)\mid X_t)}$.

Finally, to prove \eqref{eq-16}, remark that
\begin{eqnarray*}
\frac{dP_t}{dm}(X_t)&=&\frac{dP_t}{dR_t}(X_t)=E_R \left(\frac{dP}{dR}\mid X_t\right)
	=E_R(f_0(X_0)g_1(X_1)\mid X_t)\\
	&\overset{\checkmark}=& E_R(f_0(X_0)\mid X_t) E_R(g_1(X_1)\mid X_t)
	=: f_t(X_t)g_t(X_t)
\end{eqnarray*}
where we used the Markov property of $R$ at the marked equality.
\end{proof}

\subsection*{Forward and backward generators}

Let $Q\in\MO$ be a Markov measure. Its \emph{forward stochastic derivative} $\partial+\Lf^Q$ is defined by 
$$
[\partial_t+\Lf^Q_t](u)(t,x):=\lim _{h\downarrow 0}h ^{-1} E_Q\big(u(t+h,X _{t+h})-u(t,X_t)\mid X_t=x\big)
$$ 
for any measurable function $u:\iX\to\RR$ in the set $\dom\Lf^Q$ for which this limit exists $Q_t\ae$ for all $0\le t< 1$. In fact this definition is only approximate, we give it here as a support for understanding the relations between the forward and backward generators. For a precise statement see \cite[\S 2]{Leo11b}. Since the time reversed $Q^*$ of $Q$ is still Markov, $Q$ admits a \emph{backward stochastic derivative}  $-\partial+\Lb^Q$ which  is defined by 
$$
[-\partial_t+\Lb^Q_t]u(t,x):=\lim _{h\downarrow 0} h ^{-1} E_Q\big(u(t-h,X _{t-h})-u(t,X_t)\mid X_t=x\big)
$$ 
for any measurable function $u:\iX\to\RR$ in the set $\dom\Lb^Q$ for which this limit exists $Q_t\ae$ for all $0< t\le 1$. Remark that 
$
\Lb^Q_t=\Lf ^{Q^*}_{1-t},$ $0\le t\le 1.$

It is proved in  \cite[\S 2]{Leo11b} that these stochastic derivatives are extensions of the extended forward and backward generators  of $Q$ in the sense of semimartingales. In particular, they offer us a natural way for computing  generators. Later on, we shall call $\Lf^Q$ and $\Lb^Q$ generators, rather than stochastic derivatives.

For simplicity, we denote $\Lf^R=\Lb^R=L$ without the superscript $R$ and without the time arrows, since $R$ is assumed to be reversible. We also write $\Af=\Lf^P$ and $\Ab=\Lb^P$ the generators of the $(f_0,g_1)$-transform $P$ defined by  \eqref{eq-17}.

Stochastic derivatives have been introduced by E.~Nelson in \cite{Nel67} while studying the dynamical properties of the Brownian motion. The above definition (more precisely the one of \cite{Leo11b}), which is an extension of Nelson's one, is necessary for technical reasons.

\subsection*{The dynamics of the $(f,g)$-transform}

To give the expressions of the generators $\Af$ and $\Ab,$ we need to introduce the carré du champ of $R$. It is defined  for any functions $u,v$ on $\XX$ such that $u,v$ and $uv$ are in $\dom L$,    by
\begin{equation*}
\Gamma(u,v):= L(uv)-uLv-vLu.
\end{equation*}
In general, the forward and backward generators $(\partial_t+\Af_t)_{0\le t\le1}$ and $(-\partial_t+\Ab_t)_{0\le t\le1}$ of $P$  depend explicitly on $t$. The following informal statement  is known for long in specific situations.  In the  important examples which are discussed below at Section \ref{sec-standard} below, these claims are easy consequences of It\^o's formula (for instance see  \cite[Ch.\,8,\S\,3]{RY99} in the continuous diffusion case).

\begin{statement}\label{res-10}
Under some hypotheses on $R$,  the forward and backward generators of $P$ are given for any function $u:\iX\to\RR$ belonging to some class $\mathcal{U}_R$ of regular functions, by
\begin{equation}\label{eq-19}
\left\{
\begin{array}{lcll}
\Af_t u(x)&=&\displaystyle{L u(x)+\frac{\Gamma(g_t,u)(x)}{g_t(x)}},& (t,x)\in [0,1)\times\XX\\
\Ab_t u(x)&=&\displaystyle{L u(x)+\frac{\Gamma(f_t,u)(x)}{f_t(x)}},& (t,x)\in (0,1]\times\XX
\end{array}
\right.
\end{equation}
where $f_t,g_t$ are defined at \eqref{eq-20}.
\\
Because of \eqref{eq-16}, for any $t$ no division by zero occurs $P_t\ae$ 
\end{statement}

Rigorous statement and proof are given in \cite{Leo11b} for instance.

\begin{proof}[Idea of  proof of Statement \ref{res-10}] To obtain the forward generator of $P$, we are going to compute the stochastic derivative $\Af_t u(x):=\lim _{h\downarrow 0}h ^{-1}E_P[u(X _{t+h})-u(X_t)\mid X_t=x].$ Let us denote for simplicity $F_0=f_0(X_0),$ $G_t=g_t(X_t)$ and $U_t=u(X_t)$. We have 
\begin{eqnarray*}
E_P[U _{t+h}-U_t\mid X_t=x]
	&=&\frac{E_R[F_0(U _{t+h}-U_t)G_1\mid X_t=x]}{E_R[F_0G_1\mid X_t=x]}\\&=&\frac{E_R[F_0\mid X_t=x]E_R[(U _{t+h}-U_t)G_1\mid X_t=x]}{E_R[F_0\mid X_t=x]E_R[G_1\mid X_t=x]}\\
&=& \frac{E_R[(U _{t+h}-U_t)G_1\mid X_t=x]}{g_t(x)},
\end{eqnarray*}
where the Markov property of $R$ is used at second identity.
But,
\begin{eqnarray*}
&&E_R[(U _{t+h}-U_t)G_1\mid X_t=x]\\
	&=&E_R[(U _{t+h}-U_t)G _{t+h}\mid X_t=x]\\
	&=&g_t(x)E_R[(U _{t+h}-U_t)\mid X_t=x]+E_R[(U _{t+h}-U_t)(G _{t+h}-G_t)\mid X_t=x],
\end{eqnarray*}
where the first equality is a martingale identity.
We conclude  by means of the definition of $L$: 
$\lim _{h\downarrow 0}h ^{-1}E_R[(U _{t+h}-U_t)\mid X_t=x]=:Lu(x)$, and  with the following identity $\lim _{h\downarrow 0}h ^{-1}E_R[(U _{t+h}-U_t)(G _{t+h}-G_t)\mid X_t=x]=\Gamma(u,g_t)(x).$
\end{proof}

One sees that it is necessary that the functions $f_t$ and $g_t$ are regular enough  to be in the domains of the carré du champ operators. For instance, choosing  
$f_0,g_1\in\dom L$ insures that $f\in \dom(-\partial+L)$ and $g\in \dom(\partial+L)$   and also that $f$ and $g$ are \emph{classical solutions} of the following parabolic PDEs
\begin{equation}\label{eq-18}	
\left\{
\begin{array}{ll}
(-\partial_t +L) f(t,x)=0, &0<t\le 1,\\
f_0,	&t=0,
\end{array}\right.
\qquad
\left\{
\begin{array}{ll}
(\partial_t +L) g(t,x)=0, &0\le t<1,\\
g_1,	&t=1.
\end{array}\right.
\end{equation}
Even better, since $R$ is assumed to be $m$-reversible, its Markov generator $L$ is self-adjoint on $L^2(m)$ and for any $f_0,g_1$ in $L^2(m)$ we have $f\in \dom(-\partial+L)$ and $g\in \dom(\partial+L)$.
\\
It is worthwhile describing the dynamics  \eqref{eq-19} in terms of
\begin{equation}\label{eq-56}
\left\{
\begin{array}{l}
\varphi:=\log f,\\ \psi:=\log g.
\end{array}
\right.
\end{equation}
Remark  that because of \eqref{eq-16}, for any $0\le t\le 1,$ $\varphi_t$ and $\psi_t$ are well defined $P_t\ae$
In analogy with the Kantorovich potentials which appear in the optimal transport theory, we call $\varphi$ and $\psi$ the \emph{Schr\"odinger potentials}. They are  solutions of the ``second order''\footnote{When $R$ is a random walk on a discrete space for instance, then the term ``second order'' is only justified in analogy with the continuous case, see \eqref{eq-34}.} Hamilton-Jacobi equations 
\begin{equation}\label{eq-22a}
\left\{
\begin{array}{ll}
(-\partial_t +B) \varphi(t,x)=0, &0<t\le 1,\quad P_t\ae\\
\varphi_0=\log f_0,	&t=0,
\end{array}\right.
\end{equation}
and \begin{equation}\label{eq-22b}
\left\{
\begin{array}{ll}
(\partial_t +B) \psi(t,x)=0, &0\le t<1,\quad P_t\ae\\
\psi_1=\log g_1,	&t=1,
\end{array}\right.
\end{equation}
where the non-linear operator $B$ is defined by
\begin{equation*}
B u:= e ^{-u}L e^u
\end{equation*}
for any function $u$ such that $e^u\in\dom L$.

\section{Standard examples}\label{sec-standard}
We present two well-known reference processes: the reversible Brownian motion and a reversible random walk on a graph. We also apply  the above general results to these important examples. 

\subsection*{Reversible Brownian motion}

The reversible Brownian motion $R$ on $\XX=\RR^n$  is specified by
\begin{equation}\label{eq-51}
\left\{
\begin{array}{l}
L=\Delta/2,\\
R_0(dx)=m(dx)=dx
\end{array}
\right.
\end{equation}
where the Markov generator $L=\Delta/2$ is defined on $\mathcal{C}^2(\RR^n)$. It is easily checked that $R$ is $m$-reversible. 

Let $P=f_0(X_0)g_1(X_1)\,R\in\PO$ be any $(f,g)$-transform of  $R$.
By the regularity improving property of the heat kernel,  $(f_0,g_1)$ is such that $f\in\dom(-\partial_t+\Delta/2)$ for  $t\in(0,1]$ and $g\in\dom(\partial_t+\Delta/2)$ for  $t\in [0,1).$
We have
$Bu=\Delta u/2+|\nabla u|^2/2,$ $\Gamma(u,v)=\nabla u\cdot\nabla v$ for any $u,v\in \mathcal{C}^2(\XX).$ The expressions \eqref{eq-19}
\begin{equation}\label{eq-26}
\left\{\begin{array}{lcl}
\Af_t&=&\Delta/2+\nabla \psi_t\cdot\nabla\\
\Ab_t&=&\Delta/2+\nabla \varphi_t\cdot\nabla
\end{array}\right.
\end{equation}
of the forward and backward generators tell us that the density $\mu_t(x):=d P_t/dx$ solves the following parabolic PDEs
\begin{equation*}	
\left\{
\begin{array}{ll}
(\partial_t-\Delta/2) \mu_t(x)+\nabla\cdot(\mu_t\nabla \psi_t)(x)=0, &(t,x)\in (0,1]\times\XX\\
\mu_0,	&t=0,
\end{array}\right.
\end{equation*} 
where $\psi$ solves \eqref{eq-22b}:
\begin{equation}\label{eq-28}
\left\{
\begin{array}{ll}
(\partial_t +\Delta/2) \psi_t(x)+|\nabla \psi_t(x)|^2/2=0, &(t,x)\in [0,1)\times\XX\\
\psi_1=\log g_1,	&t=1,
\end{array}\right.
\end{equation}
and in the reversed sense of time
\begin{equation*}	
\left\{
\begin{array}{ll}
(-\partial_t-\Delta/2) \mu_t(x)+\nabla\cdot(\mu_t\nabla\varphi_t)(x)=0, &(t,x)\in [0,1)\times\XX\\
\mu_1,	&t=1,
\end{array}\right.
\end{equation*}
where $\varphi$  solves \eqref{eq-22a}:
\begin{equation*}	
\left\{
\begin{array}{ll}
(-\partial_t +\Delta/2) \varphi_t(x)+|\nabla \varphi_t(x)|^2/2=0, &(t,x)\in (0,1]\times\XX\\
\varphi_0=\log f_0,	&t=0.
\end{array}\right.
\end{equation*} 
It is important to note the smoothing effect of the semigroup of $R$ which allows us  to define the \emph{classical gradients} $\nabla \psi_t$ and $\nabla \varphi_t$ for all $t$ in $[0,1)$ and $(0,1]$ respectively.
They are  the forward and backward drift vector fields of the canonical process under $P.$ Also recall that, as a direct consequence of \eqref{eq-16}, no logarithm of zero is taken, $P_t$-almost surely, i.e.\ almost everywhere, and we have the time-reversal formula
\begin{equation*}
\nabla \psi_t+\nabla \varphi_t=\nabla\log \mu_t,\quad 0<t<1.
\end{equation*}

\subsubsection*{Back to Schr\"odinger problem}

Under the assumption of Theorem \ref{res-12},   $P=f_0(X_0)g_1(X_1)\,R$ solves \eqref{Sdyn} with the prescribed marginals given at \eqref{eq-21}. It is  a $(f,g)$-transform of $R,$ and we have just seen that there exist $\varphi$ and $\psi$ such that its forward and backward generators are given by \eqref{eq-26}.

\subsubsection*{Minimal action}

We derive the analogue of Benamou-Brenier formula \eqref{eq-38}.
Consider the problem of minimizing the average kinetic action
\begin{equation}\label{eq-27}
\int _{\ii\times \RR^n}\frac{|v_t(x)|^2}{2}\, \nu_t(dx)dt \to \textrm{min}
\end{equation}
among all  $(\nu,v)$ where $\nu=(\nu_t)_{0\le t\le 1}$ is a measurable path in $\mathrm{P}(\RR^n),$ $\big(v_t(x)\big) _{(t,x)\in \ii\times\RR^n}$ is a measurable $\RR^n$-valued vector field and the following constraints are satisfied:
\begin{equation}\label{eq-29}
\left\{\begin{array}{l}
(\partial_t-\Delta/2)\nu+\nabla\scal \big(\nu v\big)=0,\quad \textrm{on } (0,1)\ttimes \RR^n\\
\nu_0=\mu_0,\ \nu_1=\mu_1
\end{array}\right.
\end{equation}
where this evolution equation is meant in the following weak sense: for any  function $u\in \mathcal{C}^{1,2}_o((0,1)\times \RR^n),$ we have $\int _{(0,1)\times\RR^n}(\partial_t +\Delta/2+v\cdot\nabla)u(t,x)\,\nu_t(dx)dt=0.$

\begin{proposition}\label{res-13}
Let $\mu_0, \mu_1\in \mathrm{P}_2(\RR^n)$ be such that $H(\mu_0|\mathrm{Leb}), H(\mu_1|\mathrm{Leb})<\infty.$ Then, 
 \eqref{sdyn} has a unique solution  $\Ph\in\PO$ and the unique solution to the minimal action problem \eqref{eq-27} is $((\mu_t)_{t\in\ii},\nabla \psi)$ where 
$$
\mu_t=\Ph_t,\quad t\in\ii,
$$
and 
$$\psi_t(x)=\log E_R[g_1(X_1)\mid X_t=x],\quad (t,x)\in[0,1)\ttimes\RR^n$$ with $g_1$  a solution of \eqref{eq-13}. Moreover, $\psi$ is the unique classical solution of the Hamilton-Jacobi-Bellman equation \eqref{eq-28} and
\begin{eqnarray}\label{eq-30}
&&\inf \left\{H(\pi|R _{01});\pi\in\PXX: \pi_0=\mu_0,\pi_1=\mu_1\right\}-H(\mu_0|m)\nonumber\\
&=& \inf \left\{\IX H(\pi^x\mid R^x_1)\,\mu_0(dx);(\pi^x)_{x\in\XX}: \IX \pi^x(\cdot)\,\mu_0(dx)=\mu_1\right\}\nonumber\\
&=&  \int _{\ii\times \RR^n}\frac{|\nabla \psi_t(x)|^2}{2}\, \mu_t(dx)dt \\
&=& \inf \left\{\int _{\ii\times \RR^n}\frac{|v_t(x)|^2}{2}\, \nu_t(dx)dt ; (\nu,v): (\nu,v) \textrm{ satisfies }\eqref{eq-29}\right\}. \nonumber
\end{eqnarray}
where $R^x_1:=R(X_1\in\cdot|X_0=x)\in\PX$ and $(\pi^x)_{ x\in\XX}\in\PX ^{\XX}$ is a measurable Markov kernel.
\end{proposition}

\begin{proof}
The integrability assumptions on $\mu_0$ and $\mu_1$ ensure that the hypotheses of Proposition \ref{res-04} are satisfied. Therefore, \eqref{sdyn} admits a unique solution.
\\
As a consequence of Girsanov's theory, for any  $Q\in\PO$  such that $H(Q|R)<\infty,$ there exists some predictable $\Rn$-valued drift field $\beta$ such that 
\begin{enumerate}[(i)]
\item
$E_Q\Iii|\beta_t|^2\,dt<\infty$,
\item
$Q$ solves the martingale problem associated with the forward generator\\ $(\partial_t+\Delta/2+\beta_t\cdot\nabla)_{0\le t\le 1}$
\item
and $H(Q|R)$ is given by
\begin{equation}\label{eq-46xx}
H(Q|R)=H(Q_0|m)+E_Q\Iii\frac{|\beta_t|^2}{2}\,dt.
\end{equation}
\end{enumerate}
For a proof with an analytic flavour, see for instance \cite{Leo11a}.

It is proved in \cite{CL95} (see also \cite{Mika90} for a related result) that when $(\nu,v)$ satisfies $(\partial_t-\Delta/2)\nu+\nabla\scal \big(\nu v\big)=0$ on  $(0,1)\times \RR^n,$ and $\int _{\ii\times \RR^n}|v_t(x)|^2/2\, \nu_t(dx)dt<\infty,$ then there exists a path measure $Q\in\PO$  such that 
\begin{equation}\label{eq-35}
Q_t=\nu_t,\quad \forall 0\le t\le1,
\end{equation}
$Q$ solves the martingale problem associated with the forward stochastic derivative $(\partial_t+\Delta/2+v_t\cdot\nabla)_{0\le t\le 1}$ and, by  \eqref{eq-46xx}
\begin{equation}\label{eq-47}
H(Q|R)=H(\mu_0|m)+\int _{\ii\times \RR^n}\frac{|v_t(x)|^2}{2}\, \nu_t(dx)dt<\infty
\end{equation}
with $Q_0=\mu_0$ and $Q_1=\mu_1.$ Therefore, to any $(\nu,v)$ which satisfies\\ $\int _{\ii\times \RR^n}|v_t(x)|^2/2\, \nu_t(dx)dt<\infty$ and \eqref{eq-29},  one can associate some $Q\in\PO$ which  verifies \eqref{eq-35} and \eqref{eq-47}.
 Since
\begin{equation*}
\inf \eqref{S}=\inf \eqref{Sdyn}=H(\Ph|R)=H(\mu_0|m)+\IiX \frac{|\nabla \psi_t(x)|^2}{2}\,\mu_t(dx)dt
\end{equation*}
with $\mu_t:=\Ph_t$ and $\psi$ given at \eqref{eq-26}, and since $\Ph$ is the unique solution to \eqref{Sdyn}, we have proved that identity \eqref{eq-30} holds true,  $(\mu,\nabla \psi)$ solves \eqref{eq-27} and  $\mu$ is  unique in the sense that if $(\nu,v)$ and $(\nu',v')$ are solutions, then $\nu=\nu'=\mu.$ It remains to check that $\nabla \psi$ is also unique. The following property of $\mu$ holds with $v=\nabla \psi:$ 
$$
\int _{(0,1)\times\RR^n}(\partial_t +\Delta/2+v\cdot\nabla)u\,d\mu_tdt=0,\quad\forall u\in \mathcal{C}_o^\infty((0,1)\times \RR^n).
$$ 
Suppose that it is also verified with $v'.$ Subtracting, we obtain 
$
\int _{(0,1)\times\RR^n}(v'-\nabla \psi)\cdot\nabla u\,d\mu_tdt=0,$ $\forall u\in \mathcal{C}_o^\infty((0,1)\times \RR^n),
$
meaning that $v'-\nabla \psi$ is orthogonal in $L^2_{\RR^n}((0,1)\times\RR^n,d \mu_tdt)$ to $\nabla \mathcal{C}_o^\infty((0,1)\times \RR^n).$ It follows that the squared norm $\|v'\|^2:=\int _{\ii\times \RR^n}|v'_t(x)|^2\, \mu_t(dx)dt$ in $L^2_{\RR^n}((0,1)\times\RR^n,d \mu_tdt)$ is minimal at the orthogonal projection of $\nabla \psi$ on the closure of $\nabla \mathcal{C}_o^\infty((0,1)\times \RR^n).$ Of course, this projection is $\nabla \psi$ itself.
\end{proof}

Clearly, there is also a backward version of Proposition \ref{res-13}.

\subsection*{Reversible random walk on a graph} Let $R$ be a random walk on a countable connected graph $(\XX,\sim)$ where $x\sim y$ means that $x$ and $y$ are next neighbours. Its generator is given for any finitely supported function $u$ by
\begin{equation}\label{eq-52}
Lu(x)=\sum _{y:y\sim x}[u(y)-u(x)]\,J_x(y),\quad x\in\XX
\end{equation}
where  $J_x(y)>0$ for all $x\sim y$ is interpreted as the average frequency of jumps from $x$ to $y.$ Its dual formulation is the current equation
\begin{equation*}
\partial_t \rho_t(x)=\sum _{y:y\sim x}[\rho_t(y)J_y(x)-\rho_t(x)J_x(y)],
\qquad x\in\XX, 0<t<1
\end{equation*}
where $\rho_t(x):=R(X_t=x).$ Clearly, the path measure $R$ admits the stationary measure $m\in\MX$ if and only if $\sum _{y:y\sim x}[m_yJ_y(x)-m_xJ_x(y)]=0,$ for all $x\in\XX.$ It admits $m$ as a reversing measure if this global equilibrium property is reinforced into the following detailed one:
\begin{equation*}
m_xJ_x(y)=m_yJ_y(x ),\quad \forall x,y\in\XX, x\sim y.
\end{equation*}
The special case where the jump measure 
$$J_x:=\sum _{y:y\sim x}J_x(y)\,\delta_y\in\MX$$ is equal to $J^o_x=\frac{1}{n_x}\sum _{y:y\sim x}\delta_y$ where 
$$
n_x:=\# \left\{y:y\sim x\right\}<\infty 
$$ 
is the number of neighbours of $x$ which is assumed to be finite for all $x,$ corresponds to the \emph{simple random walk} on the graph $(\XX,\sim).$ It is easily checked that its reversing measure is $m^o=\sum _{x\in\XX}n_x \delta_x\in\MX$ which is unbounded when $\XX$ is infinite.
\\
For simplicity we assume in the general case that 
\begin{equation*}
\sup _{x\in\XX}J_x(\XX)<\infty
\end{equation*}
where $J_x(\XX):=\sum _{y:y\sim x}J_x(y)$ is the global average frequency of jumps at $x$. This ensures that for any initial marginal $R_0\in\MX,$ there exists a Markov measure $R\in\MO$ with generator $L$. Moreover, any bounded function $u:\XX\to\RR$ is in the domain of $L.$
\\
Viewing $L=(L_{xy})_{x,y\in\XX}$ as a  matrix with $L _{xy}=\left\{\begin{array}{ll}
J_x(y), &\textrm{if }x\sim y;\\
-J_x(\XX), &\textrm{if }x=y;\\
0,&\textrm{otherwise}
\end{array}
\right.$ and the functions as column vectors indexed by $\XX,$ we observe that the solutions of the heat equations \eqref{eq-18} are
\begin{equation}\label{eq-23}
f_t=e ^{tL}f_0\qquad \textrm{and}\qquad g_t=e ^{(1-t)L}g_1,\quad 0\le t\le 1.
\end{equation}
It follows that for any couple $(f_0,g_1)$ of nonnegative functions and any $0\le t\le 1,$ $f_t\in\dom(-\partial_t+L)$ and $g_t\in\dom (\partial_t+L).$
\\
 We have $Bu(x)=Lu(x)+\sum _{y:y\sim x}\theta(u(y)-u(x))\,J_x(y)$ with 
\begin{equation}\label{eq-24}
\theta(a):=e^a-a-1,\quad a\in\RR,
\end{equation}
$\Gamma(u,v)(x)=\sum _{y:y\sim x}[u(y)-u(x)][v(y)-v(x)]\,J_x(y),$ for any bounded functions $u,v$ and any $x\in\XX.$ Let $P=f_0(X_0)g_1(X_1)\,R$ be any $(f,g)$-transform of the random walk $R.$ Applying \eqref{eq-19}, the forward and backward generators of $P$ turn out to be the jump generators associated respectively with the jump measures
\begin{equation*}
\left\{\begin{array}{lclcl}
\overrightarrow{J}_x&=&\displaystyle{\sum _{y:y\sim x} \frac{g_t(y)}{g_t(x)}J_x(y)\,\delta_y}
&=&\displaystyle{\sum _{y:y\sim x} \exp \big(\psi	_t(y)-\psi_t(x)\big)J_x(y)\,\delta_y}\\
\overleftarrow{J}_x&=&\displaystyle{\sum _{y:y\sim x} \frac{f_t(y)}{f_t(x)}J_x(y)\,\delta_y}
&=&\displaystyle{\sum _{y:y\sim x} \exp \big(\varphi	_t(y)-\varphi_t(x)\big)J_x(y)\,\delta_y}
\end{array}\right.
\end{equation*}
Again, no division by zero occurs $P_t\ae$ for every $0\le t\le1,$ i.e.\ everywhere for each $0<t<1.$ The functions $f$ and $g$ satisfy \eqref{eq-23} and the Schr\"odinger potentials $\psi$ and $\varphi$ satisfy \eqref{eq-22b}:
\begin{equation}\label{eq-34}
\left\{
\begin{array}{ll}
(\partial_t +L) \psi_t(x)+\sum _{y:y\sim x}\theta\big(\psi_t(y)-\psi_t(x)\big)\,J_x(y)=0, &(t,x)\in [0,1)\times\XX\\
\psi_1=\log g_1,	&t=1,
\end{array}\right.
\end{equation}
and \eqref{eq-22a}:
\begin{equation*}
\left\{
\begin{array}{ll}
(-\partial_t +L) \varphi_t(x)+\sum _{y:y\sim x}\theta\big(\varphi_t(y)-\varphi_t(x)\big)\,J_x(y)=0, &(t,x)\in (0,1]\times\XX\\
\varphi_0=\log f_0,	&t=0.
\end{array}\right.
\end{equation*}

\subsubsection*{Minimal action}

Let $\theta^*$ be the convex conjugate of $\theta$ defined at \eqref{eq-24}, i.e.
\begin{equation*}
\theta^*(b)=
\left\{\begin{array}{ll}
(b+1)\log(b+1)-b, & \textrm{if \ }b>-1\\
1,& \textrm{if \ }b=-1\\
+\infty,& \textrm{if \ }b<-1.
\end{array}\right.
\end{equation*}
Consider the problem of minimizing
\begin{equation}\label{eq-33}
\IiXX \theta^*\big(j(t,x;y)-1\big)\, \nu_t(dx)J_x(dy)dt\to \textrm{min}
\end{equation}
among all $(\nu,j)$ where $\nu=(\nu_t)_{0\le t\le 1}$ is a measurable path in $\PX,$ $j:\ii\times\XXX\to[0,\infty)$ is a measurable nonnegative function and the following constraints are satisfied:
\begin{equation}\label{eq-32}
\left\{\begin{array}{l}
\partial_t \nu_t(x)-\sum _{y:y\sim x}\big\{\nu_t(y) j(t,y;x) J_y(x)-
\nu_t(x) j(t,x;y)J_x(y)\big\}=0,\ \textrm{on }(0,1)\!\times\!\XX\\
\nu_0=\mu_0, \nu_1=\mu_1
\end{array}\right.
\end{equation}
where we write $\nu_t=\sum _{x\in\XX}\nu_t(x)\,\delta_x.$
 
\begin{proposition}\label{res-14}
Let $\mu_0, \mu_1\in \mathrm{P}(\RR^n)$ be such that $\inf \eqref{S}<\infty,$ for instance when the assumptions of Proposition \ref{res-04} are satisfied. 
\\
The unique solution to the minimal action problem \eqref{eq-33} is $(\mu,j^\psi)$ where $\mu=(\mu_t)_{t\in\ii}$ with
$$
\mu_t=\Ph_t,\quad t\in\ii,
$$
and $\Ph\in\PO$ is the unique solution of \eqref{Sdyn}, and $$j^\psi(t,x;y)=\exp(\psi_t(y)-\psi_t(x))$$ with
$$\psi_t(x)=\log E_R[g_1(X_1)\mid X_t=x],\quad (t,x)\in[0,1)\times\RR^n$$ with $g_1$  a solution of \eqref{eq-13}. Moreover, $\psi$ is the unique classical solution of the Hamilton-Jacobi-Bellman equation \eqref{eq-34}
and
\begin{eqnarray}\label{eq-31}
&&\inf \left\{H(\pi|R _{01});\pi\in\PXX: \pi_0=\mu_0,\pi_1=\mu_1\right\}-H(\mu_0|m)\nonumber\\
&=& \inf \left\{\IX H(\pi^x\mid R^x_1)\,\mu_0(dx);(\pi^x) _{x\in\XX}: \IX \pi^x(\cdot)\,\mu_0(dx)=\mu_1\right\}\nonumber\\
&=& \IiXX \theta^*\big(j^\psi(t,x;y)-1\big)\, \mu_t(dx)J_x(dy)dt\\
&=& \inf \left\{\IiXX \theta^*\big(j(t,x;y)-1\big)\, \nu_t(dx)J_x(dy)dt; (\nu,j): (\nu,j) \textrm{ satisfies }\eqref{eq-32}\right\}\nonumber
\end{eqnarray}
where $R^x_1:=R(X_1\in\cdot|X_0=x)\in\PX$ and $(\pi^x)_{ x\in\XX}\in\PX ^{\XX}$ is any measurable Markov kernel.
\end{proposition}
Remark that $\theta^*\big(j^\psi(t,x;y)-1\big)=\theta^*\big(\frac{g_t(y)-g_t(x)}{g_t(x)}\big)$ where $\frac{g_t(y)-g_t(x)}{g_t(x)}$ looks like a discrete logarithmic derivative which is analogous to $\nabla \psi_t(x)=\nabla\log g_t(x)$.

\begin{proof}
The proof follows the same line as Proposition \ref{res-13}'s one.
\end{proof}

\section{Slowing down}\label{sec-slow}

In this section, we describe an efficient way to recover optimal transport as a limit of Schr\"odinger problems. The main idea consists in slowing  the reference process down to a no-motion process. In the following lines, we present some heuristics and refer the reader to \cite{Leo12a} for a rigorous treatment in the case where  $\XX$ is a real vector space and \cite{Leo12c} in the alternate case where $\XX$ is a discrete graph. The specific case of the reversible Brownian motion has been investigated by T.~Mikami in \cite{Mika04} with a stochastic control approach which differs  from what is presented below.

Let $R$ be Markov with generator $L$. The slowed down process is represented by the sequence $(R^k)_{k\ge1}$ in $\MO$ of Markov measures associated with the generators $$L^k:= L/k,\qquad k\ge1.$$
Remark that slowing the process down doesn't modify its reversible measure $m$; one converges more slowly towards the same equilibrium. Suppose also that the sequence $(R^k)_{k\ge1}$ in  $\MO$ obeys the large deviation principle in $\OO$ with speed $\alpha_k$ and rate function $C$, meaning approximately   that for a  ``large class'' of measurable subsets $A$ of $\OO,$ we have
\begin{equation}\label{eq-53}
R^k(A)\underset{k\rightarrow
\infty}{\asymp} \exp\left[-\alpha_k\inf_{\omega\in A} C(\omega)\right].
\end{equation}
For instance, in the case \eqref{eq-51} when $R$ is the reversible Brownian motion on $\Rn$, Schilder's theorem tells us that $C$ is the kinetic action \eqref{eq-43} and $\alpha_k=k.$ In the case \eqref{eq-52} when $R$ is a reversible random walk, it is proved in  \cite{Leo12c} that $\alpha_k=\log k$ and the rate function is
\begin{equation}\label{eq-65}
C(\omega):= \sum _{0\le t\le 1} \1 _{\{\omega _{t^-}\not=\omega_t\}},\qquad \omega \in \OO,
\end{equation}
where in this situation $\OO$ is the space of all right-continuous paths with finitely many jumps. Remark that $C$ is simply the total number of jumps of the path.

At a heuristic level, the $\Gamma$-convergence of 
\begin{equation}\label{Skdyn}
H(P|R^k)/\alpha_k\to \textrm{min};\qquad P\in\PO:P_0=\mu_0,P_1=\mu_1
\tag{S$^k _{\textrm{dyn}}$}
\end{equation}
as $k$ tends to infinity to
\begin{equation}\label{MKdyn}
\IO C\,dP\to \textrm{min};\qquad P\in\PO:P_0=\mu_0,P_1=\mu_1
\tag{MK$_{\textrm{dyn}}$}
\end{equation}
is best seen with the dual problems. Without getting into the details, to show this $\Gamma$-convergence, it is enough to prove that the objective functions of the dual problems converge pointwise, see \cite{Leo12a} for the details. Let us check this pointwise convergence. Recall that the dual problem of \eqref{s} is
\begin{equation}\label{d}
\Delta_R(\varphi_0,\psi_1)\to \textrm{max};\qquad \varphi_0,\psi_1\in C_b(\XX)
\tag{D}
\end{equation}
with $\Delta_R(\varphi_0,\psi_1):=\IX \varphi_0\,d \mu_0+\IX \psi_1\,d \mu_1
-\log \IO e ^{\varphi_0(X_0)+\psi_1(X_1)}\,dR.$
Consequently, for each $k\ge1,$ the dual problem of \eqref{skdyn} consists in maximizing $\alpha_k ^{-1}\Delta _{R^k}(\varphi_0,\psi_1)$ or equivalently: $\alpha_k ^{-1}\Delta _{R^k}(\alpha_k\varphi_0,\alpha_k \psi_1).$ This leads us to
\begin{equation}\label{dk}
\IX \varphi_0\,d \mu_0+\IX \psi_1\,d \mu_1
-\alpha_k ^{-1}\log \IO e ^{\alpha_k[\varphi_0(X_0)+\psi_1(X_1)]}\,dR^k\to \textrm{max};\quad \varphi_0,\psi_1\in C_b(\XX)
\tag{D$^k$}
\end{equation}
The pointwise limit, as $k$ tends to infinity, of $\alpha_k ^{-1}\Delta _{R^k}(\alpha_k\varphi_0,\alpha_k \psi_1)$ is a direct consequence of the large deviation principle \eqref{eq-53} and the Laplace-Varadhan integral lemma \cite[Thm.\,4.3.1]{DZ}, which provide us with 
\begin{equation*}
\Lim k \alpha_k ^{-1}\log \IO e ^{\alpha_k[\varphi_0(X_0)+\psi_1(X_1)]}\,dR^k=
\sup_\OO \left\{\varphi_0(X_0)+\psi_1(X_1)-C\right\}.
\end{equation*}
Here, we took advantage of $\varphi_0(X_0)+\psi_1(X_1)\in C_b(\OO)$ to apply the Laplace-Varadhan  lemma. We see that the pointwise limit of the objective function of \eqref{dk} is 
$$
\Lim k \alpha_k ^{-1}\Delta _{R^k}(\alpha_k\varphi_0,\alpha_k \psi_1)=
\IX \varphi_0\,d \mu_0+\IX \psi_1\,d \mu_1 -\sup_\OO \left\{\varphi_0(X_0)+\psi_1(X_1)-C\right\}
$$
and the limit dual problem is (this is  informal)
\begin{equation*}
\IX \varphi_0\,d \mu_0+\IX \psi_1\,d \mu_1 -\sup_\OO \left\{\varphi_0(X_0)+\psi_1(X_1)-C\right\}\to \textrm{max};\qquad \varphi_0,\psi_1\in C_b(\XX).
\end{equation*}
But this problem is equivalent to
\begin{equation}\label{dinf}
\IX \varphi_0\,d \mu_0+\IX \psi_1\,d \mu_1\to \textrm{max};\qquad \varphi_0,\psi_1\in C_b(\XX):\varphi_0\oplus \psi_1\le c
\tag{D$^\infty$}
\end{equation}
where 
\begin{equation}\label{eq-66}
c(x,y):= \inf\left\{C(\omega); \omega\in\OO:\omega_0=x,\omega_1=y\right\} ,
\end{equation}
recall \eqref{eq-46}. To see this, first remark that for any $\lambda\in\RR,$ transforming $\varphi_0\oplus \psi_1$ into $\varphi_0\oplus \psi_1+\lambda$ doesn't modify the value of $\IX \varphi_0\,d \mu_0+\IX \psi_1\,d \mu_1 -\sup_\OO \left\{\varphi_0(X_0)+\psi_1(X_1)-C\right\}$. Hence, when $\sup_\OO \left\{\varphi_0(X_0)+\psi_1(X_1)-C\right\}<\infty,$ we can normalize $\varphi_0\oplus \psi_1$ so that $\sup_\OO \left\{\varphi_0(X_0)+\psi_1(X_1)-C\right\}=0$ and we obtain the equivalent problem
\begin{equation*}
\IX \varphi_0\,d \mu_0+\IX \psi_1\,d \mu_1\to \textrm{max};\qquad  \varphi_0,\psi_1\in C_b(\XX):\varphi_0(X_0)+\psi_1(X_1)-C\le 0.
\end{equation*}
But, $\varphi_0(X_0)+\psi_1(X_1)-C\le 0$ on $\OO$ if and only if $\varphi_0\oplus \psi_1\le c$ on $\XXX.$ 
\\
We have informally shown that $\Lim k \eqref{dk}=\eqref{dinf}$ (in some insufficiently  specified sense) which is the usual Kantorovich problem, dual to \eqref{mk}. Consequently,  we must have $\Lim k\eqref{skdyn}= \eqref{mk}$ in some sense.
Similarly, the static analogue of \eqref{skdyn} which is
\begin{equation}\label{Sk}
H(\pi|R^k _{01})/\alpha_k\to \textrm{min};\qquad \pi\in\PXX:\pi_0=\mu_0,\pi_1=\mu_1
\tag{S$^k$}
\end{equation}
converges to
\begin{equation}\label{MK}
\IXX c\,d \pi\to \textrm{min};\qquad \pi\in\PXX: \pi_0=\mu_0, \pi_1=\mu_1
\tag{MK}
\end{equation}
with $c$ as above.
\\
It happens that this convergence is in terms of $\Gamma$-convergence.

\begin{statement}[See \cite{Leo12a,Leo12c}]\label{res-15}
Suppose that the slowed down Markov measure $R^k\in\MO$ associated with the generator $L^k:=L/k$ satisfies the large deviation principle \eqref{eq-53} with speed $\alpha_k$ and rate function $C$ in $\OO.$ Then,
\begin{equation*}
\Glim k \eqref{skdyn}=\eqref{mkdyn} 
\qquad \textrm{and} \qquad
\Glim k \eqref{sk}=\eqref{mk}.
\end{equation*}
In particular:
\begin{enumerate}
\item
In the reversible Brownian motion case \eqref{eq-51} in $\XX=\Rn$,  we have:\\ $\alpha_k=k,$ $C(\omega)=\Iii |\dot\omega_t|^2/2\,dt,$ $c(x,y)=|y-x|^2/2$ and we recover \eqref{eq-50}.
\item
In the case of a random walk on a graph \eqref{eq-52},  we have:\\ $\alpha_k=\log k,$ $C(\omega)=\sum _{0\le t\le1}\1 _{\left\{\omega_t\not=\omega _{t^-}\right\}}$ and $c=d _{\sim}:$ the graph distance on $(\XX,\sim)$.
\end{enumerate}
\end{statement}
The diffusion case is treated in details in \cite{Leo12a}. In the specific Brownian case (1), the Schr\"odinger problem converges to the quadratic Monge-Kantorovich problem and, as already remarked at \eqref{eq-54}, the bridges converge as follows: $$\Lim k R ^{k,xy}=\delta _{\gamma ^{xy}}\in\PO$$ where $\gamma ^{xy}_t=(1-t)x+ty$, $0\le t\le1$ is the constant speed geodesic path between $x$ and $y.$ In case  \eqref{mkdyn} has a unique solution $\Ph,$ we also have $\Lim k \Ph^k=\Ph\in\PO.$

Since the rigorous version of the Informal Statement \ref{res-15} is simpler to state in the second case (2) of a random walk on a graph,  we refer the reader to \cite{Leo12a} for the details about (1) and we restrict our attention to (2). In this random walk case, the rigorous version of the Informal Statement \ref{res-15} is stated below at Theorem \ref{res-18}. Some preparation is needed. In particular, let us recall basic facts about $\Gamma$-convergence.

\subsubsection*{$\Gamma$-convergence}

 Recall that $\Glim k f^k=f$ on the metric space $Y$ if and only if for any $y\in Y,$
\begin{enumerate}[(a)]
\item
$\Liminf k f^k(y_k)\ge f(y)$ for any convergent sequence $y_k\to y,$
\item
$\Lim k f^k(y^o_k)=f(y)$ for some sequence $y^o_k\to y.$
\end{enumerate}
A  function $f$ is said to be coercive if for any $a\ge \inf f,$ $\left\{f\le a\right\} $ is a compact set.
\\
 The sequence $(f^k)_{k\ge1}$ is said to be equi-coercive if for any real $a$, there exists some compact set $K_a$ such that $\cup_k\left\{f^k\le a\right\} \subset K_a.$
\\
If in addition to $\Glim k f^k=f$, the sequence $(f^k)_{k\ge1}$ is equi-coercive,
 then:
\begin{itemize}
\item
$\Lim k \inf f^k=\inf f,$
\item
if $\inf f<\infty,$
any limit point $y^*$ of a sequence $(y^*_k)_{k\ge1}$ of approximate minimizers i.e.: $f^k(y^*_k)\le \inf f^k+\epsilon_k$ with $\epsilon_k\ge0$ and $\Lim k \epsilon_k=0,$ minimizes $f$ i.e.: $f(y^*)=\inf f.$
\end{itemize}
For more details about $\Gamma$-convergence, see \cite{DalMaso} for instance.

The convex indicator $\iota_A$ of  any subset $A,$ is defined to be equal to $0$ on $A$ and to $\infty$ outside $A.$
We denote for each $k\ge2,$ (we drop $k=1$ not to divide by $\log (1)$ below),
\begin{equation*}
I^k(P):=H(P|R^k)/\log k+\iota_{\{P:P_0=\mu_0,P_1=\mu_1\}},\qquad P\in\PO,
\end{equation*}
so that \eqref{Skdyn} is simply: ($I^k\to \textrm{min}).$ We also define
$$
I(P)=E_P C+\iota_{\{P:P_0=\mu_0,P_1=\mu_1\}},\qquad P\in\PO
$$
with $C$ given  at \eqref{eq-65}. The  dynamic Monge-Kantorovich problem \eqref{MKdyn} rewrites as $(I\to \textrm{min})$.
\\
Similarly, we denote for each $k\ge2,$ 
\begin{equation*}
J^k(\pi):=H(\pi|R^k _{01})/\log k+\iota_{\{\pi:\pi_0=\mu_0,\pi_1=\mu_1\}},\qquad \pi\in\PXX,
\end{equation*}
so that \eqref{Sk} is simply: ($J^k\to \textrm{min}).$ We also define
$$
J(\pi)=\IXX c\,d \pi+\iota_{\{\pi:\pi_0=\mu_0,\pi_1=\mu_1\}},\qquad \pi\in\PXX
$$
with $c$ given  at \eqref{eq-66}. The   Monge-Kantorovich problem \eqref{MK} rewrites as $(J\to \textrm{min})$.

The spaces $\PO$ and $\PXX$ are equipped with the topologies of narrow convergence:  weakened by the spaces of all numerical continuous and bounded functions. The $\Gamma$-convergences on $\PO$ and $\PXX$ are related to these topologies.

\begin{theorem}[\cite{Leo12c}]\label{res-18}
Assume that the random walk $R$ and the prescribed marginal measures $\mu_0,\mu_1\in\PX$ satisfy the hypotheses of Theorem \ref{res-17}.\\
For each $k\ge2,$ let $\Ph^k\in\PO$ and $\ph^k\in\PXX$ be the respective solutions of \eqref{Skdyn} and \eqref{Sk}.
\begin{enumerate}
\item
The sequence $(I^k)_{k\ge2}$ is equi-coercive and $\Glim k I^k=I$ in $\PO$.
\\
In particular, $\Lim k \inf \eqref{Skdyn}=\inf \eqref{MKdyn}$ and any limit point of $(\Ph^k)_{k\ge2}$ is a solution of \eqref{MKdyn}.
\\
A more careful study allows to show that there is a unique limit point $\Ph$, so that $\Lim k\Ph^k=\Ph\in\PO$, and that $\Ph$ is the only solution of the auxiliary entropy minimization  \eqref{Stdyn} which is stated below.
\item
The sequence $(J^k)_{k\ge2}$ is equi-coercive and $\Glim k J^k=J$ in $\PXX$.
\\
In particular, $\Lim k \inf \eqref{Sk}=\inf \eqref{MK}$ and any limit point of $(\ph^k)_{k\ge2}$ is a solution of \eqref{MK} which is the Monge-Kantorovich problem associated with the metric cost  $c=d_\sim:$ the usual graph distance .
\\
Furthermore, this sequence admits the unique limit point $\Ph _{01}$, so that $\Lim k\ph^k=\Ph _{01}\in\PXX$.
\end{enumerate}
\end{theorem}

It is also proved in \cite{Leo12c} that for any distinct $x,y\in\XX,$ 
\begin{equation*}
\Lim k R ^{k,xy}=\Rt ^{xy}\in\PO
\end{equation*}
where $\widetilde R ^{xy}$ is the $(x,y)$-bridge of
\begin{equation*}
\Rt:= \1 _{G}\exp \left(\Iii J _{X_t}(\XX)\,dt\right) R\in\MO
\end{equation*}
with $G$ the set of all geodesic paths on $(\XX,d_\sim).$ Remark that the set $G ^{xy}$ of all geodesic paths between any two distinct states $x$ and $y$ is  infinite since the instants of jump are not  specified: only the ordered enumeration of the  visited states is relevant.
Let us denote $\mathcal{M}(\mu_0,\mu_1)\subset\PXX$ the set of all solutions of the Monge-Kantorovich problem \eqref{mk} with $c=d_\sim,$ and introduce the subsequent auxiliary entropic minimization problem
\begin{equation*}\label{Stdyn}
H(P|\Rt)\to \mathrm{min};\qquad P _{01}\in \mathcal{M}(\mu_0,\mu_1)
\tag{$\widetilde{\mathrm{S}}_{\mathrm{dyn}}$}
\end{equation*}
The set of all solutions of \eqref{mkdyn} consists of all $P\in\PO$ concentrated on $G,$ i.e.\ $P(G)=1,$ and such that the  endpoint  marginal $P _{01}\in\PXX$ solves \eqref{mk}. Although \eqref{mkdyn} has always infinitely many solutions (for any distinct $x,y,$ $G ^{xy}$ is  infinite), the sequence of Schr\"odinger problems \eqref{skdyn} selects a unique limit point: 
\begin{equation*}
\Lim k\Ph^k=\Ph\in\PO
\end{equation*}
where $\Ph$ is the unique solution of \eqref{Stdyn}. We obtain the corresponding  results about the static problems \eqref{Sk} and \eqref{MK} by considering the push-forward mapping $P\in\PO\to P _{01}\in\PXX$.

\section{The statistical physics motivation of Schr\"odinger's problem}

 We consider a large number $n$ of independent (non-interacting) moving random particles in the state space $\XX.$ They are described by the independent stochastic processes $Y^1,\dots,Y^n$ taking their random values in $\OO$ with the laws 
\begin{equation}\label{eq-45}
\mathrm{Law}(Y^i)=R(\cdot\mid X_0=y^i_0)\in\PO,
\qquad 1\le i\le n
\end{equation}
 where $R\in\MO$ is a  path  measure and $y^i_0\in\XX$ is the deterministic initial position of the $i$-th particle.  It is also assumed that the particles are indistinguishable. Therefore, one doesn't loose information considering the empirical probability measure 
\begin{equation*}
L^n:=\frac{1}{n}\sum _{1\le i\le n}\delta _{Y^i}
\end{equation*}
which is a random element of $\PO.$ At each time $t,$ the empirical measure of the particle system is the following random element of $\PX,$ 
\begin{equation*}
L^n_t:=\frac{1}{n}\sum _{1\le i\le n}\delta _{Y^i_t},\quad 0\le t\le1.
\end{equation*}
Suppose that the initial positions are close to a profile $\mu_0\in\PX,$ i.e.
\begin{equation*}
L^n_0=\frac{1}{n}\sum _{1\le i\le n}\delta _{y^i_0} \ \underset{n\to \infty}{\rightarrow}\ \mu_0\in\PX
\end{equation*}
with respect to the narrow topology $\sigma(\PX,C_b(\XX)).$ The law of large numbers tells us that, as $n$ tends to infinity, $L^n$ converges in law to the deterministic limit $^{\mu_0} R:= \IX R(\cdot\mid X_0=x)\,\mu_0(dx)$ in $\PO$ and in particular that at time $t=1,$
\begin{equation*}
L^n_1(dy) \ \underset{n\to \infty}{\rightarrow}\ 
(^{\mu_0} R)_1:= \IX R(X_1\in dy\mid X_0=x)\,\mu_0(dx).
\end{equation*}
Schr\"odinger addressed the following problem. Suppose that at the final time $t=1,$ you observe the system in a profile $L^n_1$ far away from the expected profile $(^{\mu_0} R)_1:$ for all large enough $n,$ $L^n_1$ is in a very small neighbourhood of some $\mu_1\in\PX$ which doesn't contain $(^{\mu_0} R)_1.$  This may happen since $n$ is finite, but this is a very rare event, i.e.\ with an exponentially small probability, see \eqref{eq-48} below. \emph{Nevertheless, conditionally on this rare event, what is the most likely dynamical behaviour of the whole random system described by $L^n?$}

Before stating this rigorously at Problem \ref{prob-01} below, take a metric $\mathrm{d}$ on $\PX$ compatible
with the narrow topology  and
denote $B(\mu,\epsilon)=\{\nu\in\PX;
\mathrm{d}(\nu,\mu)<\epsilon\}$ the open ball centred at $\mu$ with radius $\epsilon>0.$

\begin{problem}[Schr\"odinger's question\ \cite{Sch32}]\label{prob-01}
Let $\mu_0, \mu_1\in\PX$ be given.
What is the limit
\begin{equation}\label{eq-S02}
    \lim_{\epsilon\downarrow 0}\Lim n
    \mathrm{Prob}(L^n\in\cdot|L^n_1\in
    B(\mu_1,\epsilon))
\end{equation}
in $\mathrm{P}(\PO)?$
\end{problem}

\subsection*{Solving the problem without getting into details}

Schr\"odinger's approximate proof contains the main ideas. It is
based on a statistical physics approach. The main tool for obtaining
the limiting behavior as $n$ tends to infinity of the combinatoric
terms is Stirling's formula. As  pointed out by F\"ollmer in \cite{Foe85}, its modern counterpart, which is
available in a much more general setting, is
Sanov's theorem.

\begin{statement}[Informal statement of Sanov's theorem]
Let $Y^1,\dots,Y^n,\dots$ be a sequence of independent  identically distributed $\OO$-valued random variables with common law $R\in\PO$\footnote{We take a probability measure, rather than $R\in\MO,$ for the simplicity of exposition.}. Define $L^n:=\frac{1}{n}\sum _{i=1}^n \delta _{Y^i}$ its empirical measure. Then, for a ``large class'' of measurable subsets $A$ of $\PO,$ we have
\begin{equation}\label{eq-S03}
    \mathrm{Prob}(L^n\in A)\underset{n\rightarrow
\infty}{\asymp} \exp\left[-n\inf_{P\in A} H(P|R)\right].
\end{equation}
\end{statement}

The rigorous statement of this result is in
terms of a large deviation principle. It is valid for a general class of spaces $\OO$, not necessarily a path space. For a comprehensive
introduction to the theory of large deviations including Sanov's theorem, a  good textbook is \cite{DZ}. One says that $H(\cdot|R)$ is the rate
function of the large deviations of
$\{L^n\}_{n\ge1}$ as $n$ tends to infinity.

\begin{proof}[Idea of proof (a hint to agree with this statement)] We consider informally the situation where $\OO$ is replaced by a three-point set. 
Take $\Omega=\{a,b,c\},$
$R=\alpha\delta_a+\beta\delta_b+\gamma\delta_c$ and
$P=p\delta_a+q\delta_b+r\delta_c$ with
$\alpha,\beta,\gamma,p,q,r>0$ and $\alpha+\beta+\gamma=p+q+r=1.$
Then,
\begin{eqnarray*}
    &&\mathrm{Prob}(L^n\approx P)\\
    &=&\mathrm{Prob}(L^n(a)\approx p,L^n(b)\approx q, L^n(c)\approx
    r)\\
  &\approx& \frac{n!}{(np)!(nq)!(nr)!}\alpha^{np}\beta^{nq}\gamma^{nr} \\
  &\approx& \exp[n\log n-np\log(np)-nq\log(nq)-nr\log(nr)+np\log\alpha+nq\log\beta+nr\log\gamma] \\
  &=& \exp[-n(p\log(p/\alpha)+q\log(q/\beta)+r\log(r/\gamma))] \\
  &=& \exp[-nH(P|R)]
\end{eqnarray*}
where we used Stirling's formula: $k!\approx \exp[k\log k-k]$ as $k$
tends to infinity.
\end{proof}

This hint is very much in the spirit of Schr\"odinger's derivation
in \cite{Sch32}. 

Since $P\mapsto H(P|R)\in[0,\infty]$ is
strictly convex and $H(P|R)=0$ if and only if $P=R$, see \eqref{eq-06}, one observes that if $R\in A,$ (\ref{eq-S03}) leads to the law of large
numbers: $\Lim n L^n=R,$ with an exponential rate of convergence.

We need a slight modification of Sanov's theorem.

\begin{statement}[Informal statement of Sanov's modified theorem]
\quad \\
Let $Y^1,\dots,Y^n,\dots$ be the sequence of independent  $\OO$-valued random variables specified by \eqref{eq-45}. Define $L^n:=\frac{1}{n}\sum _{i=1}^n \delta _{Y^i}$ its empirical measure. Then, for a ``large class'' of measurable subsets $A$ of $\PO,$ we have
\begin{equation}\label{eq-S04}
    \mathrm{Prob}(L^n\in A)\underset{n\rightarrow
\infty}{\asymp} \exp\left[-n\inf_{P\in A: P_0=\mu_0}\Big\{H(P|R)-H(\mu_0|R_0)\Big\} \right].
\end{equation}
\end{statement}
This statement is proved by D.~Dawson and J.~G\"artner in \cite[Thm.\,3.5]{DG87} where the Schr\"odinger problem, when $R$ is a diffusion process,  is re-discovered\footnote{No reference to the original papers by Schr\"odinger is given in \cite{DG87}.} and investigated by means of large deviations of large non-interacting particle systems. More precisely, \cite[Thm.\,3.5]{DG87} only states a variational formula in the spirit of  for the rate function \eqref{eq-07b}.  An alternate proof is given in  \cite[Thm.\,2.1]{CL95} which states that the rate function is $H(P|\mu_0R)+ \iota _{\left\{P:P_0=\mu_0\right\}},$ where $\mu_0R:=\frac{d \mu_0}{dR_0}(X_0)R(\cdot|X_0).$ But it is easily seen with the additive property \eqref{eq-10} of the relative entropy that for any $P$ such that $P_0=\mu_0,$ we have 
$H(P|\mu_0 R)=\IX H\big(P(\cdot|X_0=x)|R(\cdot|X_0=x)\big)\, \mu_0(dx)$
 and
$H(P|R)=H(\mu_0|R_0)+\IX H\big(P(\cdot|X_0=x)|R(\cdot|X_0=x)\big)\, \mu_0(dx),$ which implies that $H(P|\mu_0R)=H(P|R)-H(\mu_0|R_0),$ the desired result.
\\
The conditional probability in (\ref{eq-S02}) has the form
$\mathrm{Prob}(L^n\in A|L^n\in C_\epsilon)$ where
$$C_\epsilon=\{P\in\PO;P_0=\mu_0, P_1\in
B(\mu_1,\epsilon)\}$$ and $\epsilon>0$ is introduced to guarantee
that $\mathrm{Prob}(L^n\in C_\epsilon)$ doesn't vanish. With
\eqref{eq-S04} one sees that for each $\epsilon>0$ and ``all''
$A\in\PO,$
\begin{equation*}
    \mathrm{Prob}(L^n\in A|L^n\in C_\epsilon)\underset{n\rightarrow
\infty}{\asymp} \exp\left[-n\Big\{\inf_{P\in A\cap C_\epsilon}
H(P|R)-\inf_{P\in C_\epsilon} H(P|R)\Big\}\right]
\end{equation*}
Some analytical work (formally, think of $A$ as an arbitrarily
small neighbourhood of a generic $P\in\PO$)  allows us to show that this
implies that
    $\Lim n  \mathrm{Prob}(L^n\in \cdot|L^n\in
    C_\epsilon)=\delta_{\Ph_\epsilon}$
where $\Ph_\epsilon$ is the solution of the convex minimization
problem
\begin{equation*}
    H(P|R)\to \textrm{min};\qquad P\in C_\epsilon.
\end{equation*}
Note that this problem admits a unique solution since $H(\cdot|R)$
is a \emph{strictly} convex function on the convex set
$C_\epsilon.$ Existence is obtained as usual showing that
$H(\cdot|R)$ has compact sublevel sets. Finally, as $\epsilon$
decreases to zero, $C_\epsilon$ decreases to
$C=\{P\in\PO;P_0=\mu_0,P_1=\mu_1\}$ and the
objective functions of the minimization problems on $C_\epsilon:$
$H(P|R)+\iota_{C_\epsilon}(P)$ where
$\iota_{C_\epsilon}(P)=\left\{\begin{array}{ll}
  0 & \textrm{if }P\in C_\epsilon \\
  \infty & \textrm{otherwise} \\
\end{array}\right.$ increase towards $H(P|R)+\iota_{C}(P).$ Together with some compactness, this
monotonicity allows to prove easily that
$\lim_{\epsilon\rightarrow 0}\Ph_\epsilon=\Ph$ where $\Ph$ is the
unique solution to the limiting minimization problem:
\begin{equation*}
H(P|R)\to \textrm{min};\qquad P\in C.
\end{equation*}
Therefore, we have informally obtained  the answer to Schr\"odinger's
question.

\begin{statement}[The answer to Schr\"odinger's question]
The limit (\ref{eq-S02}) is
\begin{equation*}
    \lim_{\epsilon\downarrow 0}\Lim n
    \mathrm{Prob}(L^n\in\cdot|L^n_1\in
    B(\mu_1,\epsilon))=\delta_{\Ph}\in \mathrm{P}(\PO)
\end{equation*}
where $\Ph$ is the unique solution to the entropy minimization
problem
\begin{equation*}
    H(P|R)\to \emph{min};\qquad P\in\PO: P_0=\mu_0, P_1=\mu_1.
    \tag{S$_{\textrm{dyn}}$}
\end{equation*}
\end{statement}

Loosely speaking, this means that conditionally on
$L^n_0\approx\mu_0$ and $L^n_1\approx \mu_1,$ the whole system $L^n$
tends in law as $n$ tends to infinity towards $\Ph.$ In fact, the
rigorous proof of this theorem  \cite[Thm.\,7.3]{Leo07d} uses large deviation principles and
shows that this convergence is exponentially fast. Therefore,
Borel-Cantelli lemma allows us to state an almost sure version of
this conditional law of large numbers.

The same line of reasoning leads to the following evaluation of the probability that the system evolves spontaneously from the prepared initial profile $\mu_0$ to the unexpected profile final profile $\mu_1:$
\begin{equation}\label{eq-48}
 \lim_{\epsilon\downarrow 0}\Lim n \frac{1}{n}\log
    \mathrm{Prob}(L^n_1\in
    B(\mu_1,\epsilon)|L^n_0\in B(\mu_0,\epsilon))=-\big\{\inf \eqref{sdyn}- H(\mu_0|m)\big\} .
\end{equation}

These considerations show that  solving Schr\"odinger's problem amounts to solve
the convex minimization problem \eqref{Sdyn} which, in statistical physics,
enters the class of Boltzmann-Gibbs conditioning principles. 

For a variation on this theme, with killed particles, see \cite{DGW90}.

\subsection*{The  lazy gas experiment}

In his textbook \cite[pp.\,445-446]{Vill09}, C.~Villani writes in a section entitled \emph{``A fluid mechanics feeling for Ricci curvature - The lazy gas experiment"}, the following sentences.  

\emph{Take a perfect gas in which particles do not interact, and ask him to move from a certain prescribed density field at time $t=0,$ to another prescribed density field at time $t=1.$ Since the gas is lazy, he will find a way to do so that needs a minimal amount of work (least action principle). Measure the entropy\footnote{Here, the entropy is  standard Boltzmann's one: $p\mapsto-H(p|\mathrm{vol}),$ which is a concave function.} of the gas at each time, and check that it always lies \emph{above} the line joining the final and initial entropies. If such is the case, then we know that we live in a nonnegatively curved space.}

This is clearly Schr\"odinger's thought experiment. As  \cite{Vill09} is only concerned with optimal transport, this lazy gas experiment must be understood  at the level of the displacement interpolations. It refers to the important discovery by K.T.~Sturm and M.~von~Renesse  \cite{SvR05} that entropy along  displacement interpolations enjoys convexity properties related to Ricci lower bounds\footnote{The decisive milestones on the way towards this result are \cite{McC97,OV00,CMS01}.}. Namely, Otto's heuristic calculus  (see \cite{JKO98,Otto01} and \cite[Ch.\,15]{Vill09}) allows us to guess that, along the displacement interpolations $[\mu_0,\mu_1]$ with respect to quadratic optimal transport on a Riemannian manifold $\XX,$ the second derivative of the entropy as a function of time: $t\in\ii\mapsto h _{[\mu_0,\mu_1]}(t):=H(\mu_t|\vol)$ satisfies
\begin{equation}\label{eq-58}
h _{[\mu_0,\mu_1]}''(t)= \langle \Gamma_2(\psi_t),\mu_t\rangle,\quad 0\le t\le1,
\end{equation}
where $\psi$ solves \eqref{eq-41}  and $\Gamma_2(\psi):=L\Gamma(\psi)-2 \Gamma(\psi,L \psi)$, with $L=\Delta/2,$ is the iterated carré du champ.
Bochner's formula, relates $\Gamma_2$ and the Ricci curvature:
\begin{equation*}
\Gamma_2(\psi)=\| \nabla \psi\|^2_{\mathrm{HS}}+\mathrm{Ric}(\nabla\psi).
\end{equation*}
This is the starting point of the Lott-Sturm-Villani theory \cite{St06a,St06b,LV09}. 

Schr\"odinger problem suggests a slight (more realistic :-) variant of this thought experiment where displacement interpolations are replaced with entropic interpolations, see \eqref{eq-55}. This is really a lazy gas experiment, while in some sense, the above mentioned lazy gas experiment  in \cite{Vill09} is a \emph{very} lazy gas experiment. Indeed, in the displacement interpolation setting, not only the particles need to find a cooperative lazy behaviour (the transport mapping $x\mapsto y$) but also each individual particle must find an economic way to travel (the minimizing geodesic path) as a result of its intrinsic laziness: it is very slow and at the limit $k\to \infty,$ it doesn't want to move at all, recall Statement \ref{res-15}. 

It is interesting to know that, without slowing down, along the entropic interpolation $[\mu_0,\mu_1],$
\begin{equation}\label{eq-57}
h _{[\mu_0,\mu_1]}''(t)=\ud \langle \Gamma_2(\varphi_t)+\Gamma_2(\psi_t),\mu_t\rangle,\quad 0\le t\le1,
\end{equation}
where the functions $\varphi$ and $\psi$ are given at \eqref{eq-56}: $\varphi_t(z)=\log E_R(f_0(X_0)\mid X_t=z)$, $\psi_t(z)=\log E_R(g_1(X_1)\mid X_t=z)$ with $f_0, g_1$ such that $\mu_t=(f_0(X_0)g_1(X_1)\,R)_t$. This is proved by the author in \cite{Leo12d}. Remark that \eqref{eq-57} is an extension of \eqref{eq-58} where  $\psi=-\varphi.$ Unlike this asymptotic case $(k=\infty)$ where the convexity estimate \eqref{eq-58} is obtained informally by means of Otto's heuristic calculus, the entropic interpolation formula \eqref{eq-57} is rigorous since $h _{[\mu_0,\mu_1]}$ is  genuinely second differentiable on $(0,1)$ and the stochastic calculus result \eqref{eq-19} is rigorous.

In the general setting of a reference reversible measure $R$, and in particular with a reversible random walk  on a graph, it is shown in \cite{Leo12d} that stochastic calculus for the $(f,g)$-transforms as developed at Statement \ref{res-10} leads us to the following rigorous formula
\begin{equation}\label{eq-59}
h _{[\mu_0,\mu_1]}''(t)=\ud \langle \Theta_2(\varphi_t)+\Theta_2(\psi_t),\mu_t\rangle,\quad 0< t<1,
\end{equation}
where
\begin{equation}\label{eqd-10}
\Theta _{2} \psi:=L \Theta \psi
	+e ^{-\psi}\Gamma\left(e^\psi,\Theta \psi\right)
	+ e ^{-\psi}\Gamma(e^\psi,\psi)B \psi -e ^{-\psi}\Gamma(e^\psi B \psi,\psi)
\end{equation}
with
\begin{equation*}
\Theta \psi:= e ^{-\psi}\Gamma(e ^{\psi},\psi)-B \psi+L\psi.
\end{equation*}
In the special case where $L$ is a diffusion operator, then $\Theta=\Gamma$ and $\Theta_2=\Gamma_2.$ One may expect that formulae \eqref{eq-59} and \eqref{eqd-10} could  lead to some results about the curvature of graphs, in the same spirit as \eqref{eq-57}, which is related to curvature via Bochner's formula, carries information about the curvature of the underlying Riemannian manifold.

\section{A short history of Schr\"odinger's problem and related literature}

Schr\"odinger's problem was first addressed by E. Schr\"odinger in a German written article  \cite{Sch31} which was  published in 1931 and entitled 
\selectlanguage{german} 
\emph{``{\"U}ber die {U}mkehrung der {N}aturgesetze''}
\selectlanguage{english}
\footnote{On the reversibility of the laws of nature.}, then in  a French written article \cite{Sch32} which was published in 1932 and entitled \emph{``Sur la théorie relativiste de l'électron et l'interprétation de la mécanique quantique"}\footnote{On the relativistic theory of the electron and the interpretation of quantum mechanics.}. The entropy minimization problem appears at the last section VII of the 1932 article which can be read independently of the preceding sections and is entitled: \emph{``Une analogie entre la mécanique ondulatoire et quelques problèmes de probabilités en physique classique"}\footnote{An analogy between wave mechanics and some probabilistic problems in classical physics.}. 

Let us quote Schr\"odinger's introduction to this section.  

\selectlanguage{french}
\emph{
Il s'agit d'un problème classique: problème de probabilités dans la théorie du mouvement brownien. Mais en fin de compte, il \emph{ressortira} une analogie avec la mécanique ondulatoire, qui fut si frappante pour moi lorsque je l'eus trouvée, qu'il m'est difficile de la croire purement accidentelle.
\\
\`A titre d'introduction, je voudrais citer une remarque que j'ai trouvée dans les ``Glifford lectures'' de A.S. Eddington (Cambridge, 1928, p.\,216 et sqq). Eddington, en parlant de l'interprétation de la mécanique ondulatoire, fait dans une note en bas de page la remarque suivante:
\selectlanguage{english}
\footnote{This is a classical problem: a probability problem in the theory of Brownian motion. But eventually an analogy with the wave mechanics will appear. This analogy stroke me so hard once I discovered it, that it is difficult for me to believe that it is purely accidental.
\\
As an introduction, let me quote a remark that I found in the ``Glifford lectures''  of A.~S.~Eddington (Cambridge, 1928, p.\,216 et sqq). Discussing the interpretation of wave mechanics, Eddington writes in a footnote the following remark: ``The whole interpretation is very obscure, \dots '' 
}
\\
``The whole interpretation is very obscure, but it seems to depend on wether you are considering the probability \emph{after you know what has happened} or the probability for the purposes of prediction. The $\psi \bar \psi$ is obtained by introducing two symmetrical systems of $\psi$ waves travelling in opposite directions in time; one  of these must presumably correspond to probable inference from what is known (or is stated) to have been the condition at a later time.'' 
}

In 1931, wave mechanics is newly born and many physicists are puzzled by its possible interpretations. Based on Eddington's remark, one may wonder at first sight if in the quantum world knowledge from the far future is available. Of course, this is not so, but why? In his 1931-32 papers, Schr\"odinger solves this paradox by providing an amazingly close analogue of the quantum wave function propagation in the classical world, by means of the entropy minimization problem \eqref{Sdyn}. In particular,   formula \eqref{eq-16} in Theorem \ref{res-11}: $P_t(dx)=f_t(x)g_t(x)\,m(dx),$ must be interpreted as the classical analogue of Born's formula: $P_t(dx)=\psi_t(x)\bar \psi_t(x)\,dx.$  Let us quote \cite{Sch32} again (this quotation also appears  in \cite{Foe85}) to emphasize that, although derived in a heuristic manner in \cite{Sch31,Sch32}, the system \eqref{eq-13} and Born's formula \eqref{eq-16} are motivated by the following question of large deviations in the framework of the lazy gas experiment:
\selectlanguage{french}
\emph{
Imaginez que vous observez un système de particules en diffusion, qui soient en équilibre thermodynamique. Admettons qu'à un instant donné $t_0$ vous les ayez trouvées en répartition à peu près uniforme et qu'à $t_1>t_0$ vous ayez trouvé un écart spontané et \emph{considérable} par rapport à cette uniformité. On vous demande de quelle manière cet écart s'est produit. Quelle en est la manière la plus probable\,? 
\selectlanguage{english}
\footnote{Imagine that you observe a system of diffusing particles which is in thermal equilibrium. Suppose that at a given time $t_0$ you see that their repartition is almost uniform and that at $t_1>t_0$ you find  a spontaneous and \emph{significant} deviation from this uniformity. You are asked to explain how  this deviation occurred. What is its most likely behaviour?}
}

As a concluding comment in his 1932 article, Schr\"odinger writes:
\selectlanguage{french}
\emph{
La fonction [d'onde] complexe $\psi$ correspond à \emph{deux} fonctions réelles, de sorte qu'il suffit de définir les conditions aux limites en se donnant la valeur de $\psi$ à \emph{un seul} instant déterminé; c'est la façon de voir généralement admise en mécanique quantique. Est-elle la seule admissible? Dans notre problème, cela reviendrait à regarder comme données les valeurs de $f$ et $g$\footnote{With the notation of the present article.} à un instant déterminé (au lieu des valeurs de leur produit à deux instants différents), chose inadmissible et absolument dénuée de sens.
\\
Doit-on interpréter la remarque d'Eddington, citée plus haut, comme signalant la nécessité de modifier cette manière de voir en mécanique ondulatoire et prendre comme conditions aux limites les valeurs d'une seule probabilité réelle à deux instants différents?
\selectlanguage{english} 
\footnote{
The complex [wave] function $\psi$ corresponds to \emph{two} real functions. Therefore, it is enough to define the limit conditions by prescribing the value of $\psi$ at a \emph{unique} given time. This is the regular practice in quantum mechanics. Is it the only admissible one? In our problem, this would correspond to considering that the values of $f$ and $g$ [with the notation of the present article] are prescribed at a given time (instead of the values of their product at two distinct times). This is inadmissible and meaningless.
\\
Should one interpret the previously quoted remark of Eddington, as a hint for the necessity of modifying our usual way of looking at quantum mechanics by defining the limit conditions in terms of the values of a single real probability at two distinct times?
}
}

This has been performed in 1942 by R.~Feynman in his PhD thesis \cite{Fey42}, without knowing Schr\"odinger's contribution. Feynman's thesis is entitled: \emph{The principle of least action in quantum mechanics}. Based on a seminal article by Dirac \cite{Dirac33}, entitled \emph{The Lagrangian in quantum mechanics} (also reproduced  in \cite{Fey42}), and in contrast with the regular Hamiltonian approach, Feynman's thesis proposes a Lagrangian approach to quantum mechanics which will be further  developed in several directions, see \cite{FH65}.  

\subsubsection*{F\"ollmer's contribution}

Although Schr\"odinger obtains  the classical Born formula \eqref{eq-16}, he does not write explicitly the problems \eqref{Sdyn} and \eqref{S}. Their explicit formulation is due to H.~F\"ollmer in his Saint-Flour lecture notes \cite[pp.\,154-167]{Foe85}. Proposition \ref{res-03} which is based on the additive  property  of the relative entropy \eqref{eq-10}, also appears in  \cite{Foe85}.

\subsection*{Early mathematical developments}

Although this part of Schr\"odinger's work has been forgotten for some decades, it had influenced leading mathematicians  soon after its publishing.

\subsubsection*{Reciprocal processes, 1932}

Very soon after Schr\"odinger's 1931 article, S.~Bernstein published in 1932 an article \cite{Bern32} about the general problem of deriving limit theorems for sequences of dependent random variables. Among other notions, he explored the Markov property and, motivated by \cite{Sch31},  proposed a type of  time-correlation which is less restrictive than the Markov property and is still symmetric with respect to time reversal\footnote{It is not clear that Bernstein was aware of the time-symmetry of the Markov property. This symmetry has clearly been identified twenty years later by J.L.~Doob in his textbook \cite{Doob53}.}. He suggested that such stochastic processes could be called reciprocal process. While a Markov  measure $Q\in\MO$ satisfies for any $0\le s\le t\le1,$
$$Q(X _{[0,s]}\in\cdot, X _{[t,1]}\in\cdot\cdot\mid X _{[s,t]})
	=Q(X _{[0,s]}\in\cdot\mid X _s)Q(X _{[t,1]}\in\cdot\cdot\mid X_t),$$
a path measure $Q\in\MO$ is reciprocal if for any $0\le s\le t\le1,$
$$
Q(X _{[s,t]}\in\cdot\mid X _{[0,s]}, X _{[t,1]})=Q(X _{[s,t]}\in\cdot\mid X _s, X _t).
$$
Any Markov measure is reciprocal, but the converse is false. 
\\
The theory of reciprocal processes has been forgotten for a while after Bernstein's article and was eventually developed by B. Jamison  in 1974, \cite{Jam74,Jam75}. A significant contribution of Jamison to the theory of Schr\"odinger problem was that its solution $\Ph$ is not only reciprocal, but also Markov and it is indeed an $h$-transform of the Markov reference process. This is performed without any entropy, but solely by means of reciprocal transitions. F\"ollmer recovered these results in \cite{Foe85} using the entropy minimization problem \eqref{sdyn}. For more information about the relations between reciprocal and Markov measures, see \cite{LRZ12}.

\subsubsection*{Time-reversal, 1936}

In the very first lines of his celebrated paper \cite{Kol36} about Markov processes and time-reversal, A.~Kolmogorov quotes Schr\"odinger's 1931 paper as a motivation. This has been surprisingly forgotten afterwards.

\subsubsection*{Schr\"odinger system, 1940}

Schr\"odinger had left open the problem of finding criteria for the system \eqref{eq-13} to have a solution $(f_0,g_1).$ In 1940, R.~Fortet \cite{Fort40} proposed a partial solution and in 1960, A.~Beurling \cite{Beu60} gave a solution close to the statement of Theorem \ref{res-08}. Beurling's proof also relies upon an entropy argument. Beurling's result was  improved by Jamison in \cite{Jam75} who obtained the complete solution of Schr\"odinger's system.

\subsection*{Stochastic deformations of  mechanics} The aim of Euclidean quantum mechanics (EQM), which is mainly developed by J.-C.~Zambrini since 1986 \cite{Zam86,CZ91,CWZ,CZ08}, is to transfer by analogy, known results from quantum mechanics to the theory of stochastic processes and the other way round\footnote{Unlike Nelson's stochastic mechanics  \cite{Nel85} or Nagasawa's interpretation of quantum mechanics \cite{Naga00}, EQM is not aimed at giving a stochastic interpretation of quantum mechanics. Such a project still remains an open problem eighty years after the advent of this theory.}.  The starting paper \cite{Zam86} of this program relies on Schr\"odinger's  discovery and adapts Jamison's results for an appropriate class of reciprocal processes (unlike Zambrini, Jamison doesn't use the time-reversed filtration in his construction of reciprocal processes). Then the EQM program was extended to the derivation of rigorous results about various kind of  stochastic processes which are suggested by the textbook \emph{Quantum mechanics and path integrals} by Feynman and Hibbs  \cite{FH65}. This textbook presents, indeed, a time-symmetric (Lagrangian) approach to quantum mechanics  which extends Feynman's early works and in particular his PhD thesis \cite{Fey42}.

Feynman's approach is an enlightening, efficient and intuitive guideline for physicists, but unfortunately it is impossible  to put it on a rigorous mathematical ground: it is proved that Feynman's integral is an oddly defined object.
However, replacing Feynman's integration by stochastic calculus suggests interesting results about diffusion processes. The first of these results  was the celebrated Feynman-Kac's formula \cite{Kac49}. EQM viewpoint, however, is that there is much more in Feynman's method than this time-asymmetric measure theoretic perturbative formula. EQM uses Kac's strategy in a systematic manner and its basic program is to obtain rigorous stochastic analogues of several intuitive  statements from \cite{FH65}; intuitive, but highly efficient since they are corroborated by experiments. In EQM, the natural stochastic processes to work with are reciprocal processes. However, in several important situations, it appears that the critical (solving some variational problem) reciprocal processes are Markov. In this case, it is sufficient to work with $(f,g)$-transforms of Markov reference processes (see \cite{Jam75, Foe85} for an $h$-transform representation) and their extensions: $f_0(X_0)\exp \big(\int _0^1 U(X_t)\,dt\big) g_1(X_1)\, R$ with the additional Feynman-Kac integral term $(x,y)\mapsto\exp \big(\int _0^1  U(X_t)\,dt\big)\,R ^{xy}$ which is the classical analogue of Feynman's propagator. These extensions of  $h$-transforms are also used by M. Nagasawa in \cite{Naga89,Naga00} who also explores connections between  stochastic processes and quantum physics which are highly inspired by the Schr\"odinger problem.

It is also possible to stochastically deform all the mathematical tools of classical mechanics to derive new results about diffusion processes. For instance,  M.~Thieullen designed a second order calculus for reciprocal processes in \cite{Th93} and  without referring to \eqref{sdyn} or entropy in general, M.~Thieullen and J.-C.~Zambrini have obtained a stochastic deformation of Noether theorem  
\cite{TZ97a}.

\subsubsection*{An interesting problem}
This suggests that it would be also interesting to derive a type of Noether theorem for the Monge-Kantorovich dynamical problem. Let us give some hint of what is meant. In the Euclidean case, the displacement interpolation $[\mu_0,\mu_1]$ is a  solution of \eqref{MKdyn} with $C=C_{\textrm{kin}}:=\Iii |\dot X_t|^2/2\,dt$, the  kinetic action. It has a constant speed; this means that twice the average kinetic energy $\IO |\dot X_t|^2\,d\Ph=\IX |\nabla \psi_t(x)|^2\,\mu_t(dx),$ with the notation of the Benamou-Brenier formula \eqref{eq-38}, doesn't depend on time $t$. What happens when considering, instead of $C_{\textrm{kin}}$,  the  action functional $C=\Iii \big(|\dot X_t|^2/2+U(X_t)\big)\,dt$ which should be connected with some Newton equation?  What are the quantities that are conserved along the minimizer $[\mu_0,\mu_1],$ in terms of the symmetries of the potential $U?$

\subsection*{Stochastic optimal control}

Optimal transport can be  deformed into a stochastic optimal control problem. This is mainly the contribution of T.~Mikami, see \cite{Mika09} for an overview of this approach and some of its main developments. With \eqref{eq-46xx}, one obtains that the Brownian  Schr\"odinger problem \eqref{sdyn}, i.e.\ taking $R$ to be the reversible Brownian motion, is also expressed as follows:
\begin{equation}\label{eq-49}
E_{P^u} \int_0^1 L(u_t)\, dt \to \textrm{min};\quad u\in \mathcal{A}: P^u_0=\mu_0, P^u_1=\mu_1
\end{equation}
where $\mathcal{A}$ is the set (of admissible controls) which consists of all the $\Rn$-valued progressively measurable processes $u$ and $P^u$ (if it exists) is the law of the semi-martingale   
\begin{equation*}
X^u_t=X^u_0+\int_0^t u_s\,ds+W_t,\quad 0\le t\le 1
\end{equation*}
	 where $W$ is a standard Brownian motion starting from 0 and 
\begin{equation*}
L(u)=|u|^2/2,\quad u\in\Rn.
\end{equation*}
This is a stochastic version of the quadratic Monge-Kantorovich problem \eqref{mkdyn}:
\begin{equation*}
E_{P^u} \int_0^1 L(u_t)\, dt \to \textrm{min};\quad u\in \mathcal{A}_{\textrm{MK}}: P^u_0=\mu_0, P^u_1=\mu_1
\end{equation*} 
which is obtained by replacing $\mathcal{A}$ with  $\mathcal{A}_{\textrm{MK}},$ the set of all controls $u\in L^1 _{\Rn}(\ii)$ and taking $P^u$ to be the law of 
$$
X^u_t=X^u_0+\int_0 ^t u_s\,ds,\quad 0\le t\le 1,
$$ 
a process with a random initial position and a deterministic evolution. 
\\
This theory extends naturally to the case where $L$ is strictly convex, regular and coercive enough: $\lim _{|u|\to \infty}L(u)/|u|^p=\infty,$ for some $p>1. $ But results close to optimal transport are obtained with $L$ admitting a quadratic growth, i.e.\ in harmony with the Brownian motion $W.$

When $L$ is quadratic, if the Brownian motion $W$ is replaced with $\sqrt{\epsilon}W$ and $\epsilon$ tends to zero, then Mikami  
shows in \cite{Mika04} that \eqref{sdyn} tends to \eqref{mkdyn}. Unfortunately, this type of convergence remains unclear unless $L$ is  quadratic, i.e.\ unless the stochastic optimal control problem corresponds to \eqref{sdyn}.

T.~Mikami and M.~Thieullen have proved a Kantorovich-type dual equality in  \cite{MT06} for \eqref{eq-49} and recovered related optimal transport results in \cite{MT08}. T. Mikami has intensively studied the connections between stochastic control and optimal transport. In particular,  soon after the discovery by Jordan, Kinderlehrer and Otto \cite{JKO98} of the relation between gradient flows, Wasserstein distance and dissipative evolution equations, he proposed in \cite{Mika00} a stochastic  approach to the JKO approximation scheme. In addition to the already cited articles by Mikami, several other  works by the same author are related to a probabilistic approach to optimal transport:
\cite{Mika02,Mika06,Mika08,Mika12}.
Let us also quote the early contributions of Mikami 
\cite{Mika90} and P.~Dai~Pra  \cite{DP91} where the Schr\"odinger problem is translated in terms of stochastic control.

\subsection*{Penalized Monge-Kantorovich problem}

The connection between the Monge-Kantorovich and the Schr\"odinger
problems is also exploited implicitly in some works where
\eqref{mk} is penalized by a relative entropy, leading to the
minimization problem
\begin{equation*}
      \IXX c\,d\pi +\frac1k H(\pi|\rho) \rightarrow \min;\quad \pi\in\PXX : \pi_0=\mu_0,
    \pi_1=\mu_1
\end{equation*}
where $\rho\in\PXX$ is a fixed reference probability measure on
$\XXX,$ for instance $\rho=\mu_0\otimes\mu_1.$ Putting
$\rho^k(dxdy)=Z_k^{-1}e^{-kc(x,y)}\,\rho(dxdy)$ with $Z_k=\IXX
e^{-kc}\,d\rho<\infty,$ up to the additive constant $\log(Z_k)/k,$
this minimization problem rewrites as \eqref{sk} with $\rho^k$ instead of $R^k _{01}$. See for
instance the papers by R\"uschendorf and Thomsen \cite{RT93,RT98}
and the references therein. Also interesting are the  papers
by Dupuy, Galichon and Salanie \cite{GS10, DG12} with an applied point of view.

\appendix

\section{Relative entropy with respect to an unbounded measure}\label{sec-entropy}

This appendix section is a short version of \cite[\S\,2]{Leo12b} which we refer to for more details.
Let $r$ be some $\sigma$-finite positive measure on some  space $Y$. The relative entropy of the probability measure $p$ with respect to $r$ is loosely defined by
\begin{equation}\label{eq-01app}
H(p|r):=\int_Y \log(dp/dr)\, dp\in (-\infty,\infty],\qquad p\in \PY
\end{equation}
if $p\ll r$ and $H(p|r)=\infty$ otherwise. 
\\
More precisely, when $r$ is a probability measure,  we have $$H(p|r)=\int_Y h(dp/dr)\,dr\in[0,\infty],\qquad p,r\in\PY$$ with $h(a)=a\log a-a+1\ge 0$ for all $a\ge0,$ (take $h(0)=1).$ Hence,  the definition \eqref{eq-01app} is meaningful. It follows from the strict convexity of $h$ that $H(\cdot|r)$ is also strictly convex. In addition, since $h(a)=\inf h=0 \iff a=1,$ we also have for any $ p\in \PY,$
\begin{equation}\label{eq-06}
H( p|r)=\inf H(\cdot|r)=0\iff  p=r.
\end{equation}
If $r$ is unbounded, one must restrict the definition of $H(\cdot|r)$ to some subset of $\PY$ as follows. As $r$ is assumed to be $\sigma$-finite, there exist  measurable functions $W:Y\to [1,\infty)$ such that
\begin{equation}\label{eq-02}
z_W:=\int_Y e ^{-W}\, dr<\infty.
\end{equation}
Define the probability measure $r_W:= z_W ^{-1}e ^{-W}\,r$ so that $\log(dp/dr)=\log(dp/dr_W)-W-\log z_W.$ It follows that for any $p\in \PY$ satisfying $\int_Y W\, dp<\infty,$ the formula 
\begin{equation}\label{eq-03}
H(p|r):=H(p|r_W)-\int_Y W\,dp-\log z_W\in (-\infty,\infty]
\end{equation}
is a meaningful definition of the relative entropy which is coherent in the following sense. If $\int_Y W'\,dp<\infty$ for another measurable function $W':Y\to[0,\infty)$ such that $z_{W'}<\infty,$ then $H(p|r_W)-\int_Y W\,dp-\log z_W=H(p|r _{W'})-\int_Y W'\,dp-\log z_{W'}\in (-\infty,\infty]$.
\\
Therefore, $H(p|r)$ is well-defined for any $p\in \PY$ such that $\int_Y W\,dp<\infty$ for some measurable nonnegative function $W$ verifying \eqref{eq-02}. It can be proved that
\begin{eqnarray}
H(p|r)&\overset{(\textrm{i})}=&\sup \left\{\int_Y u\,dp-\log\int_Y e^u\,dr; u:Y\to[-\infty,\infty), \int_Y e^u\,dr<\infty\right\} \label{eq-07} \\
&\overset{(\textrm{ii})}=& \sup \left\{\int_Y u\,dp-\log\int_Y e^u\,dr; u\in C_W(Y)\right\},\label{eq-07b}
\end{eqnarray}
where 
\begin{enumerate}[(i)]
\item
identity (i) is valid when $p$ is assumed to be absolutely continuous with respect to $r;$
\item
identity (ii) is meaningful when $Y$ is a topological space equipped with its Borel $\sigma$-field since we have set $C_W(Y)$ to be the space of all continuous functions $u:Y\to\RR$ such that $\sup |u|/W<\infty$, where $W$ is any nonnegative function satisfying \eqref{eq-02}. In this case, it follows that, being the supremum of affine continuous functions, $H(\cdot|r)$ is a convex lower semi-continuous function with respect to the weak topology $\sigma(\{p\in\PY;\int_Y W\,dp<\infty\},C_W(Y))$. 
\end{enumerate}
Clearly, identity (i) entails that $H(p|r)=\infty$ whenever $p\in \PY$ is such that $\int_Y W\, dp =\infty$ 
\\
It follows from the strict convexity of $H(\cdot|r_W)$ and \eqref{eq-03} that $H(\cdot|r)$ is also strictly convex.

Let $Y$ and $Z$ be two  Polish spaces equipped with their Borel $\sigma$-fields. For any measurable function $\phi:Y\to Z$ and any measure $q\in \mathrm{M}_+(Y)$ we have the  disintegration formula
\begin{equation}\label{eq-09}
q(dy)=\int _{Z} q(dy|\phi=z)\, \phi\pf q(dz)
\end{equation}
where $z\in Z\mapsto q(\cdot|\phi=z)\in \mathrm{P}(Y)$ is measurable, and the following additive   property
\begin{equation}\label{eq-10}
H(p|r)=H(\phi\pf p|\phi\pf r)+\int _{Z} H\Big(p(\cdot\mid \phi=z)|r(\cdot\mid\phi=z)\Big)\,\phi\pf p(dz),
\end{equation}
 is valid for any $p\in \PY$ and any $\sigma$-finite $r\in \mathrm{M}_+(Y).$
In particular, as $r(\cdot\mid\phi=z)$ is a probability measure for each $z$, with \eqref{eq-06} we see that
\begin{equation}\label{eq-08}
H(\phi\pf p|\phi\pf r)\le H(p|r),\quad \forall p\in\PY
\end{equation}
with equality if and only if 
\begin{equation}\label{eq-11}
p(\cdot\mid \phi=z)=r(\cdot\mid\phi=z),\quad \forall z, \phi\pf p\as
\end{equation}

\selectlanguage{english} 
%\bibliographystyle{alpha}
%\bibliography{schroe.bib}

\end{document}